\documentclass[11pt, oneside]{amsart}

\newcommand{\hide}[1]{}

\usepackage{amssymb}
\usepackage{graphicx}
\usepackage{amsmath}

\usepackage{geometry}
\usepackage{caption}

\usepackage{float}

\usepackage{bm}

\usepackage{wrapfig}

\usepackage{xcolor}

\usepackage{tikz}
\usetikzlibrary{arrows, arrows.meta}

\def\textcolor#1{}

\newcommand{\R}{\mathbb{R}}

\newcommand{\Hh}{\mathbb{H}}
\renewcommand{\S}{\mathbb{S}}

\newcommand{\Ss}{\mathcal{S}}
\newcommand{\Bb}{\mathcal{B}}
\newcommand{\B}{{\bm B}}
\newcommand{\tF}{\widetilde{F}}

\newcommand{\rr}{\mathfrak{r}}

\renewcommand{\tilde}{\widetilde}
\renewcommand{\rho}{\varrho}
\renewcommand{\phi}{\varphi}
\newcommand{\eps}{\varepsilon}
\renewcommand{\theta}{\vartheta}

\DeclareMathOperator{\sn}{sn}
\DeclareMathOperator{\cn}{cn}

\DeclareMathOperator{\dist}{dist}

\newcommand{\sm}{\setminus}

\renewcommand{\ge}{\geqslant}
\renewcommand{\le}{\leqslant}

\usepackage[colorlinks=true,urlcolor=blue]{hyperref}

\theoremstyle{theorem}
\newtheorem{theorem}{Theorem}[section]
\newtheorem{lemma}[theorem]{Lemma}
\newtheorem{proposition}[theorem]{Proposition}
\newtheorem{corollary}[theorem]{Corollary}

\newtheorem{MainTheorem}{Theorem}

\newtheorem*{BorisenkoConjecture}{Conjecture}

\newtheoremstyle{Intro}
{}
{}
{\itshape}
{}
{\scshape}
{.}
{.5em}
{}

\theoremstyle{Intro}
\newtheorem*{theoremIPClassical}{\textsc{The Classical Isoperimetric Inequality}}
\newtheorem*{theoremIPBall}{\textsc{The Reverse Isoperimetric Inequality due to Ball}}
\newtheorem*{theoremRIPConcave}{The Reverse Isoperimetric Inequality for $\lambda$-concave Bodies}

\theoremstyle{definition}
\newtheorem{definition}[theorem]{Definition}

\theoremstyle{remark}
\newtheorem*{remark}{\textsc{Remark}}

\newcounter{reminder}

\newtheoremstyle{claim}
  {}
  {}
  {\itshape}
  {0pt}
  {\scshape}
  {.}
  { }
  {\thmname{#1}\thmnumber{ #2}\thmnote{ (#3)}}

\theoremstyle{claim}
\newtheorem{claim}[theorem]{Claim}

\numberwithin{equation}{section}

\title[Reverse isoperimetric problems under curvature constraints]{Reverse isoperimetric problems under curvature constraints}
\author[Kostiantyn Drach]{Kostiantyn Drach}
\author[Kateryna Tatarko]{Kateryna Tatarko}
	\thanks{\ The first author is partially supported by ERC Advanced Grant ``SPERIG'' (\#885707), Agencia Estatal de Investigaci\'on grants PID2023-147252NB-I00, CNS2025-166633, and by the Severo Ochoa and Mar\'ia de Maeztu Program for Centers and Units of Excellence in R\&D (CEX2020-001084-M); the second author is partially supported by NSERC Discovery Grant number 2022-02961. }

\date{}

\address{Departament de Matem{\`a}tiques i Inform{\`a}tica, Universitat de Barcelona, Gran Via de les Corts Catalanes 585, 08007 Barcelona, Spain}
\address{Centre de Recerca Matem{\`a}tica, Edifici C, Carrer de l'Albareda, 08193 Bellaterra, Spain}
\email{kostiantyn.drach@ub.edu}
\address{Department of Pure Mathematics, University of Waterloo, Waterloo, ON, N2L 3G1, Canada}
\email{ktatarko@uwaterloo.ca}

\begin{document}

\begin{abstract}
	In this paper we solve several reverse isoperimetric problems in the class of \emph{$\lambda$-convex} bodies, i.e., convex bodies whose curvature at each point of their boundary is bounded below by some $\lambda > 0$.
	
	We give an affirmative answer in $\mathbb{R}^3$ to a conjecture due to Borisenko which states that the \emph{$\lambda$-convex lens}, i.e., the intersection of two balls of radius $1/\lambda$, is the unique \emph{minimizer} of volume among all {$\lambda$-convex bodies} of given surface area.
	
	Also, we prove a \emph{reverse inradius inequality}: in model spaces of constant curvature and arbitrary dimension, we show that the $\lambda$-convex lens (properly defined in non-zero curvature spaces) has the smallest inscribed ball among all $\lambda$-convex bodies of given surface area. This solves a conjecture due to Bezdek on minimal inradius of isoperimetric {ball-polyhedra} in $\R^n$.  
 
\end{abstract}

\maketitle

{\hypersetup{linktoc=page}
\tableofcontents
}

\section{Introduction}
\subsection{Classical and reverse isoperimetric problems}\label{classical}

\begin{center}
\emph{What body has the largest volume for a given surface area?} 
\end{center}

This classical \emph{isoperimetric problem} has been studied, generalized and solved in a great variety of contexts, for example, in spaces of constant curvature, etc. (see, e.g., \cite{Ros}).  In the $n$-dimensional Euclidean space $\mathbb R^n$, for instance, it leads to the following fundamental inequality:  

\begin{theoremIPClassical}
Let $n \ge 2$. Let $K \subset \R^n$ be a domain with surface area $|\partial K|$ and volume $|K|$, and let $B \subset \mathbb R^n$ be a ball. If  $|\partial K| = |\partial B|$, then $|K| \le |B|$, and equality holds if and only if $K$ is a ball.
\end{theoremIPClassical}

At the same time, one may pose the following \emph{reverse isoperimetric problem}:

\begin{center}
\emph{What body has the smallest volume for a given surface area?}
\end{center}
The reverse isoperimetric problem has received a great deal of attention in recent years. Observe that for a body with a given surface area, the $n$-dimensional volume can be arbitrarily small (think of a flat pancake, i.e., an oblong and thin domain). Hence, as stated, the reverse isoperimetric problem has a trivial solution. In his celebrated work, Ball \cite{Bal1, Bal2} overcame this issue by restricting the problem to affine classes of convex bodies. In this setting, the problem turns out to be non-trivial, and Ball showed the following: 

\begin{theoremIPBall}
Let $n \ge 2.$ Let $K \subset \R^n$ be a convex body and $T \subset \R^n$ be a regular simplex. Then there is an affine image $K'$ of $K$ such that $|\partial K'| = |\partial T|$ and $|K'| \ge |T|$. 
\end{theoremIPBall}

The equality case in Ball's reverse isoperimetric inequality was settled later by Barthe \cite{Ba}. 

\subsection{Reverse isoperimetric problem under curvature constraints}

In this paper, we take a different approach and restrict the reverse isoperimetric problem to the class of convex bodies with \emph{uniform curvature constraints}.  
Namely, we require that the boundary of a convex body is either not too flat, or not too curved, in the following sense. 

Let $K$ be a convex body in an $n$-dimensional model space $M^n(c)$ of constant curvature $c \in \R$, $n \ge 2$. We say that $K$ is \emph{$\lambda$-convex} if there exists $\lambda > 0$ so that the principal curvatures $k_i$, $i \in \{1,\ldots, n-1\}$, with respect to inward pointing normals of the boundary $\partial K$ satisfy
\[
k_i(p) \ge \lambda \quad \text{ for each point }p \in \partial K.
\]
If the boundary of $K$ is not smooth, then the bound above is understood in the barrier sense, which we will make precise in Definition~\ref{Def:LambdaConvex} (see also Figure~\ref{Fig:UniformConvexity}). In a similar way, $K$ is called \emph{$\lambda$-concave} if the inequality 
\[
\lambda \ge k_i(p) \ge 0 
\]
holds at each point $p \in \partial K$ and for every $i \in \{1,\ldots,n-1\}$.

It turns out that for $\lambda$-convex and $\lambda$-concave bodies, the reverse isoperimetric problem is well-posed and has a non-trivial solution in spaces of constant curvature. In this paper, we will focus on its version in the Euclidean space. Despite the fact that the notions of $\lambda$-concavity and $\lambda$-convexity look dual, the progress of resolving the reverse isoperimetric problem for these classes is not uniform. 

For $\lambda$-concave bodies in $\R^n$, the reverse isoperimetric problem was completely solved by Chernov and the authors in \cite{CDT}, where the following sharp result was shown as part  of a whole range of \emph{reverse quermassintegral inequalities}:

\begin{theoremRIPConcave}[\cite{CDT}]
Let $n \ge 2$. Let $K~\subset~\R^n$ be a $\lambda$-concave body, and let $S \subset \R^n$ be a \emph{$\lambda$-sausage}, i.e., the convex hull of two balls of radius $1/\lambda$.  If $|\partial K| = |\partial S|$, then $|K| \ge |S|$, and equality holds if and only if $K$ is a $\lambda$-sausage.
\end{theoremRIPConcave}

\noindent An alternative proof of this result and its generalization to Minkowski relative geometry was given later in \cite{SY} using different techniques (see also \cite{Piotr} for an elegant proof of a weaker version of the inequality). 

Much less is known about the reverse isoperimetric problem for $\lambda$-convex bodies in $\R^n$, $n \ge 2$. The following conjecture is due to Borisenko (see \cite{CDT}): 

\begin{BorisenkoConjecture}[The Reverse Isoperimetric Inequality for $\lambda$-convex bodies]
Let $n \ge 2$. Let $K \subset \R^n$ be a $\lambda$-convex body, and let $L \subset \R^n$ be a $\lambda$-convex \emph{lens}, i.e., the intersection of two balls of radius $1/\lambda$ (see Figure~\ref{Fig:Lens}). If $|\partial K| = |\partial L|$, then $|K| \ge |L|$, and equality holds if and only if $K$ is a $\lambda$-convex lens.
\end{BorisenkoConjecture}
For $n=2$, the conjecture was confirmed in \cite{BorDr14} (see \cite{FKV} for an alternative proof). Further progress was made in  \cite{BorDr15_1, Dr14,BorNA} fully resolving an analogous conjecture in all 2-dimensional spaces of constant curvature and their natural generalizations (e.g., Alexandrov spaces of curvature bounded below; see also \cite{CFP} for a generalization of the inequality in $\R^2$ with additional assumptions on convexity). For $n \ge 3$, partial progress was made so far only in $\R^3$, where the conjecture was confirmed in a restrictive class of surfaces of revolution \cite{Dr3}.  

In this paper, we completely solve Borisenko's conjecture in $\mathbb{R}^3$. Our main result~is
\begin{MainTheorem}[The Reverse Isoperimetric Inequality for $\lambda$-convex bodies in $\R^3$]
	\label{mainthm}
	Let $K \subset \R^3$ be a $\lambda$-convex body and $L \subset \R^3$ be a $\lambda$-convex lens. If $|\partial K| = |\partial L|$, then $|K| \ge |L|.$ Moreover, equality holds if and only if $K$ is a $\lambda$-convex lens.
\end{MainTheorem}
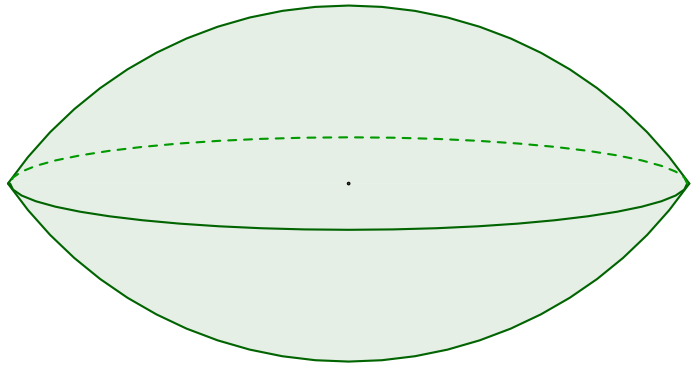
\begin{figure}[h]
\definecolor{uququq}{rgb}{0.25098039215686274,0.25098039215686274,0.25098039215686274}
\definecolor{ttzzqq}{rgb}{0.2,0.6,0.}
\begin{tikzpicture}[line cap=round,line join=round,>=triangle 45,x=1.2cm,y=1.2cm]
\clip(-4.,-2.1) rectangle (4.,2.1);
\fill[line width=1.pt,color=ttzzqq,fill=ttzzqq,fill opacity=0.10000000149011612] (-3.842986564339558,0.046330343966855855) -- (3.842986564339557,0.04633034396685584) -- (3.842986564339557,-0.04633034396685584) -- (-3.842986564339558,-0.046330343966855855) -- cycle;
\draw [shift={(0.,-3.0079099572141677)},line width=1.pt,color=ttzzqq,fill=ttzzqq,fill opacity=0.10000000149011612]  plot[domain=0.6715357772283473:2.470056876361446,variable=\t]({1.*4.908862347943025*cos(\t r)+0.*4.908862347943025*sin(\t r)},{0.*4.908862347943025*cos(\t r)+1.*4.908862347943025*sin(\t r)});
\draw [shift={(0.,0.)},line width=1.pt,color=ttzzqq]  plot[domain=-3.239166666629998:0.09757401304020531,variable=\t]({1.*3.861353355677969*cos(\t r)+0.*0.4755768507003521*sin(\t r)},{0.*3.861353355677969*cos(\t r)+1.*0.4755768507003521*sin(\t r)});
\draw [shift={(0.,0.)},line width=0.7pt,dash pattern=on 1.5pt off 2.5pt,color=ttzzqq]  plot[domain=0.09757401304020531:3.0440186405495884,variable=\t]({1.*3.861353355677969*cos(\t r)+0.*0.4755768507003521*sin(\t r)},{0.*3.861353355677969*cos(\t r)+1.*0.4755768507003521*sin(\t r)});
\draw [shift={(0.,3.0079099572141677)},line width=1.pt,color=ttzzqq,fill=ttzzqq,fill opacity=0.10000000149011612]  plot[domain=0.6715357772283473:2.470056876361446,variable=\t]({1.*4.908862347943025*cos(\t r)+0.*4.908862347943025*sin(\t r)},{0.*4.908862347943025*cos(\t r)+-1.*4.908862347943025*sin(\t r)});
\begin{scriptsize}
\draw [fill=uququq, color=uququq] (0.,0.) circle (1.0pt);
\end{scriptsize}
\end{tikzpicture}
\caption{The $\lambda$-convex lens (or the UFO-body).}
\label{Fig:Lens}
\end{figure}

For the proof of Theorem~\ref{mainthm}, we use some ideas of Nazarov \cite{Nazarov}. Our approach is different from \cite{Dr3}, where optimal control theory and Pontryagin's Maximum Principle were used to solve the corresponding constrained optimization problem. Our proof can also be adapted to work in $\R^2$, as we illustrate in Appendix~\ref{App:RIPDim2}. This gives an alternative proof of the result in~\cite{BorDr14}.

Recall that the \emph{inradius} of a convex body is the radius of the largest ball contained in the body. As a byproduct of the proof of Theorem \ref{mainthm} in Section~\ref{Sec:ProofA}, we establish the inequality $r(K) \ge r(L)$ between the inradii $r(K)$ and $r(L)$ of a $\lambda$-convex body $K \subset \R^3$ and lens $L \subset \R^3$ satisfying $|\partial K| = |\partial L|$. This is an instance of a {\it reverse inradius inequality}. As the second main result of the paper (see  Section~\ref{Sec:RII}),  we show that the same inequality holds in all constant curvature spaces of arbitrary dimension. In particular, this gives an indication that Borisenko's conjecture might be true for $n\ge 4$.

 We also refer the reader to \cite{HTr95, Ga12}  where some partial progress on the reverse isoperimetric problem under curvature constraints in the non-convex setting was made. (In the non-convex setting, the problem does not always have a well-defined solution. For example, in $\R^2$, there exists a constant $a > 0$ so that for arbitrarily large $\ell > 0$ there is a domain $D$ with area $a$ such that $\partial D$ is a simple closed curve of length $\ell$ and curvature $|k| \le 1$.)

\subsection{Reverse inradius inequality for $\lambda$-convex bodies}\label{Sec:RII}

In the spirit of our discussion in Section \ref{classical}, one can ask the following question:

\begin{center}
\emph{What convex body has the largest inradius for a given surface area?} 
\end{center}

\noindent In $\R^n$, the answer is a ball, and it is the unique body with this property. This is a classical result that immediately follows from the fact that the surface area of an inscribed ball is no greater than the surface area of the convex body in which it is inscribed. The reverse question
\begin{center}
\emph{What convex body has the smallest inradius for a given surface area?}
\end{center}
is well-defined and has a non-trivial solution for $\lambda$-convex bodies in spaces of constant curvature, while this is not the case for general convex bodies.

For $2$-dimensional $\lambda$-convex bodies in the Euclidean plane, i.e., the space of constant zero curvature, Milka \cite{MilInradius} proved that the $\lambda$-convex lens has the smallest inradius for a given length of the boundary. Much later, this result was generalized in \cite{Dr4} to arbitrary $\lambda$-convex domains in $2$-dimensional Alexandrov spaces of curvature bounded below, thus closing this question in dimension $2$. Recently, Bezdek \cite{BezdekConj} showed that among all $\lambda$-convex bodies in $\R^n$, $n\ge 2$, the $\lambda$-convex lens has the smallest inradius for a given \emph{volume}. Bezdek conjectured that the same result must be true if one considers $\lambda$-convex bodies of a given intrinsic volume \cite[Conjecture 5]{BezdekConj}. We confirm this conjecture in the case of surface area by establishing the following sharp and much more general result:

\begin{MainTheorem}[The Reverse Inradius Inequality for $\lambda$-convex bodies]
\label{Thm:MainB}
Let $n \ge 2$. Let $K \subset M^n(c)$ be a $\lambda$-convex body in a model space $M^n(c)$ of constant curvature $c$, and let $L \subset M^n(c)$ be a $\lambda$-convex lens. If $|\partial K| = |\partial L|$, then the inradii of the bodies satisfy $r(K) \ge r(L)$. Moreover, equality holds if and only if $K$ is a $\lambda$-convex lens.
\end{MainTheorem}

For the definition of a $\lambda$-convex lens in spaces of constant curvature, see Definition~\ref{Def:LambdaLens}.

The following result is an immediate corollary of Theorem~\ref{Thm:MainB} in the $2$-dimensional space of constant positive curvature that can be obtained using spherical duality (see~\cite[Section~4]{BorDr15_1}):

\begin{corollary}[The Reverse Outer Radius Inequality in $\mathbb S^2$]
Let $K \subset \S^2$ be a $\lambda$-concave body and $S \subset \S^2$ be a $\lambda$-sausage (i.e.\ the convex hull of two circles of geodesic curvature equal to $\lambda$). If $|K| = |S|$, then the circumradii $R(K), R(S)$ of the bodies (i.e.\ the radii of the smallest geodesic disks containing the bodies) satisfy $R(K) \le R(S)$. Moreover, equality holds if and only if $K$ is a $\lambda$-sausage.
\end{corollary}
\noindent It is plausible that the same result holds for $\lambda$-concave bodies in spheres of arbitrary dimension, but some additional tools might be required for the proof.

Finally, we would like to point out that the reverse inradius problem in the class of $\lambda$-concave bodies in $\R^n$ has a simple solution. Indeed, $r(K) \ge 1/\lambda$ because every $\lambda$-concave body $K \subset \R^n$ contains a ball of radius $1/\lambda$. This inequality is the best possible and is attained, in particular, by $\lambda$-sausages.

\subsection{Structure of the paper}

In Section~\ref{Sec:Background}, we give the definition of $\lambda$-convexity and provide some background on the notion. In Sections~\ref{Sec:ProofA} and \ref{Sec:ProofB} we give the proof of Theorems \ref{mainthm} and \ref{Thm:MainB}, respectively. Using our method, we provide a self-contained proof of the reverse isoperimetric inequality in $\mathbb R^2$ in Appendix~\ref{App:RIPDim2}. Finally, Appendix~\ref{App:Ass} contains an asymptotic computation in $\R^n$ showing that, as $n \to \infty$, the $\lambda$-convex lens has smaller volume for given surface area compared to its closest `competitor', a \emph{$\lambda$-convex spindle} (it will be defined in Appendix~\ref{App:Ass}).

\bigskip
\noindent
\textbf{Acknowledgement.} The authors are grateful to Fedor Nazarov and Dmitry Ryabogin for their invaluable help. The second author wants to thank ICERM for the hospitality, where part of the work on the paper was completed during her participation in the program ``Harmonic Analysis and Convexity". The authors would also like to thank the anonymous referees for their constructive comments and suggestions that led to the improvement of the manuscript, and in particular for pointing out the reference \cite{Cs} that helped to make some of our original arguments more concise. 

\section{Some background}
\label{Sec:Background}

We denote by $M^n(c)$ a complete simply connected Riemannian manifold of dimension $n \ge 2$ and of constant sectional curvature equal to $c \in \R$. We call these manifolds \emph{model spaces}. It is well-known that $M^n(c)$ is either the spherical, Euclidean, or hyperbolic space depending on whether, respectively, $c > 0$, $c=0$, or $c < 0$.  

We will use the notation $|A|$ for the $n$-dimensional volume of a measurable set $A$ in $M^n(c)$. We will use the same notation $|\partial A|$ for the $(n-1)$-dimensional volume of the boundary of $A$ (provided this volume is well-defined).

As usual, a \emph{convex body} in a model space is a compact, geodesically convex set with non-empty interior.

\subsection{Totally umbilical hypersurfaces in model spaces}	
\label{SSec:Facts}

For a given $\lambda > 0$, let $\Ss_\lambda \subset M^n(c)$ be a complete totally umbilical hypersurface of curvature $\lambda$, i.e., a hypersurface whose normal curvatures at every point and in every direction are equal to $\lambda$ with respect to some choice of normals. We will use the notation $\Bb_\lambda$ for the convex region bounded by $\Ss_\lambda$. Everywhere below we assume that $\Bb_\lambda$ is closed.

The following classification of complete totally umbilical hypersurfaces is well-known (see, e.g., \cite[Section 35.2.4]{BZ}):

\begin{enumerate}
\item
If $c = 0$, i.e., the ambient model space is the Euclidean space $\mathbb R^n$, then the sets $\Ss_\lambda$ and $\Bb_\lambda$ are, respectively, spheres and balls of radius $1/\lambda$. 
\item
If $c > 0$, i.e., $M^n(c)$ is the sphere $\S^n(c)$ of radius $1/\sqrt{c}$, then $\Ss_\lambda$ is a geodesic sphere of radius $1/\sqrt{c} \cdot \cot^{-1} (\lambda/\sqrt{c})$ and $\Bb_\lambda$ is the corresponding ball.
\item
The situation is more interesting if $c<0$, i.e., when the model space is the hyperbolic space $\Hh^n(c)$, and it depends on the relation between $c$ and $\lambda$:
\begin{enumerate}
\item
If $\lambda > \sqrt{-c}$, then $\Ss_\lambda$ is a geodesic sphere of radius $1/\sqrt{-c} \cdot \coth^{-1} (\lambda/\sqrt{-c})$, and $\Bb_\lambda$ is the geodesic ball. In this case, $\Bb_\lambda$ is compact. 
\item
If $\lambda = \sqrt{-c}$, then $\Ss_\lambda$ is a \emph{horosphere} and $\Bb_\lambda$ is a \emph{horoball}. In this case, $\Bb_\lambda$ is not compact (however, it can be compactified by adding a single point at infinity).
\item
\label{It:Horo}
Finally, if $\lambda < \sqrt{-c}$, then $\Ss_\lambda$ is an \emph{equidistant hypersurface} (or \emph{equidistant} for short), i.e., a connected component of the set of points at a given distance from a totally geodesic hyperplane. In literature, equidistants are also called \emph{hyperspheres}. By explicit computation, one can see that for each $\Ss_\lambda$ there exists a unique hyperplane $\mathcal H \subset \Hh^n(c)$ such that $\Ss_\lambda$ is a connected component of
\[
\left\{x \in \Hh^n(c) \colon \dist_{\Hh^n(c)}(x, \mathcal H) = \frac{1}{2\sqrt{-c}} \log \frac{\sqrt{-c} + \lambda}{\sqrt{-c} - \lambda}\right\}.
\] 
We will call 
\begin{equation}
\label{Eq:rr}
\rr_\lambda := \frac{1}{2\sqrt{-c}} \log\frac{\sqrt{-c}+\lambda}{\sqrt{-c} -\lambda} = \frac{1}{\sqrt{-c}} \tanh^{-1}\left(\frac{\lambda}{\sqrt{-c}}\right)     
\end{equation}
the \emph{characteristic distance} for the equidistant $\Ss_\lambda \subset \Hh^n(c)$. In the equidistant case, $\Bb_\lambda$ is not compact and there is a cone of directions to infinity within $\Bb_\lambda$.
\end{enumerate}

\noindent
In the Poincar\'e unit ball model of $\Hh^n(c)$, the hypersurfaces $\Ss_\lambda$ are Euclidean spheres or hyperplanes that: lie inside the unit ball if $\lambda > \sqrt{-c}$, touch the unit sphere if $\lambda = \sqrt{-c}$, and intersect the unit sphere at an angle different from $\pi/2$ if $\lambda < \sqrt{-c}$. Note that when the angle of intersection tends to $\pi/2$, the corresponding equidistant hypersurface tends to a totally geodesic hyperplane.
\end{enumerate}

\subsection{$\lambda$-convex bodies}	
\label{SSec:LConv}
Fix some $\lambda>0$. Recall that for every totally umbilical hypersurface $\Ss_\lambda$ in a model space, we denote by $\Bb_\lambda$ the corresponding convex region bounded by~$\Ss_\lambda$.

\begin{definition}[$\lambda$-convex body]
\label{Def:LambdaConvex}
A convex body $K \subset M^n(c)$ is \emph{$\lambda$-convex} if for each $p \in \partial K$ there exists a neighborhood $U_p \subset M^n(c)$ and a totally umbilical hypersurface $\Ss_\lambda$ that passes through $p$ in such a way that
\[
U_p \cap \partial K \subseteq \Bb_\lambda
\]
(see Figure~\ref{Fig:UniformConvexity}).
\end{definition}

\begin{figure}
	\includegraphics[scale=0.45]{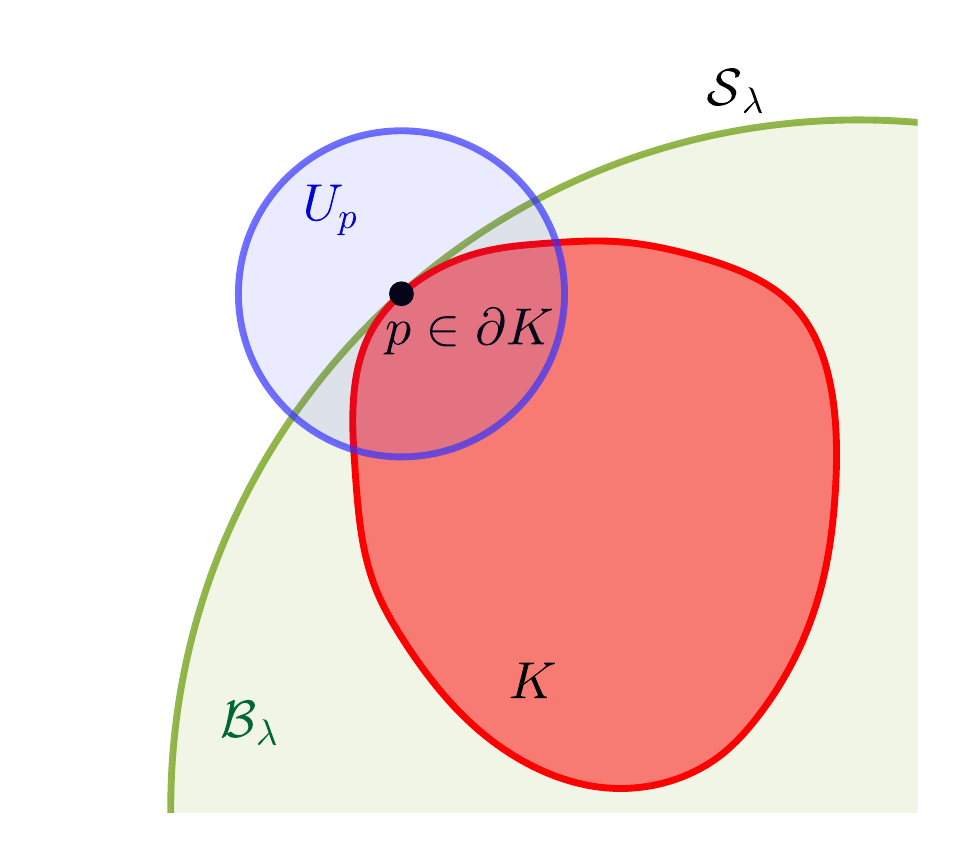}
	\caption{$K$ is a convex body, $\Ss_\lambda$ is a totally umbilical hypersurface of constant normal curvature $\lambda > 0$, e.g., a sphere. $K$ is \emph{$\lambda$-convex} if this local picture holds in some neighborhood $U_p$ around every point $p \in \partial K$. In this case, we say that the principal curvatures of $\partial K$ are bounded below by $\lambda$ in the \emph{barrier sense}.}
	\label{Fig:UniformConvexity}
\end{figure}

For a $\lambda$-convex body $K$, we call $\Bb_\lambda = \Bb_\lambda(p)$ a \emph{supporting convex set of curvature $\lambda$} (passing through $p \in \partial K$). If $\Bb_\lambda$ is a geodesic ball, we will call it a \emph{supporting ball} (and its boundary $\Ss_\lambda$ a \emph{supporting sphere}) at the point $p$. 

The following classical theorem of Blaschke turns the local condition in Definition~\ref{Def:LambdaConvex} into a global one. It was first proven in \cite{Bla56} for $\lambda$-convex bodies with smooth boundary in $\R^2$, and was later extended in \cite{Rau74, BrStr89} to general $\lambda$-convex domains in $\R^n$, and in \cite{Kar68, Mil70, How99, BDr15, DrBla} to $\lambda$-convex bodies in other model spaces (and more generally, in some Riemannian manifolds of bounded curvature; see \cite{DrBla} and references therein). 

\begin{theorem}[Blaschke's Rolling Theorem]
\label{Thm:Bla}
Let $K \subset M^n(c)$ be a $\lambda$-convex body. Then for every $p \in \partial K$ there exists a supporting convex set $\Bb_{\lambda}(p)$ of curvature $\lambda$ such that $K  \subseteq  \Bb_{\lambda}(p)$. 
\end{theorem}

Finally, let us point out an equivalent definition of $\lambda$-convexity which is, in some sense, dual to Definition~\ref{Def:LambdaConvex}. The equivalence can be seen using Theorem~\ref{Thm:Bla}. A convex body $K \subset M^n(c)$ is said to be $\lambda$-convex if and only if any two points in $K$ can be connected by a shortest arc of a curve of constant geodesic curvature $\lambda$ and the whole spindle hypersurface, that is the union of all such arcs connecting the two points, lies inside $K$ (this justifies the name ``spindle convexity'' used by some authors; see Section~\ref{SSec:LConvPol}). From this point of view, $\lambda$-convexity is a generalization of the standard notion of convexity where we say that a body is convex if and only if the line segment that connects any two of its points lies entirely inside the body. Spindle hypersurfaces will appear in our discussion in Appendix~\ref{App:Ass}.

\subsection{$\lambda$-convex polytopes and their properties}	
\label{SSec:LConvPol}

In this subsection, we define a special class of $\lambda$-convex bodies.

\begin{definition}[$\lambda$-convex polytope]
	\label{Def:LambdaConvexPol}
A \emph{$\lambda$-convex polytope} in $M^n(c)$ is a $\lambda$-convex body $K$ given as the intersection of finitely many convex regions $\Bb_\lambda$. The boundary $\partial K$ is piece-wise smooth, and the closure of each maximal connected smooth part is called a \emph{facet} of $K$. A facet is of dimension $n-1$ and is equal to the intersection $K \cap \partial \Bb_\lambda$ for some $\Bb_\lambda$.

For a point $p \in \partial K$, let $F_i$, $i \in I_p$, be all the facets of $K$ containing $p$, and let $E_p$ be the connected component of $\bigcap_{i \in I_p} F_i$ containing $p$. We say that $p$ is a \emph{vertex of $K$} if $E_p$ is $0$-dimensional (and hence $E_p = \{p\}$). If $E_p$ is $1$-dimensional, then it is called an \emph{edge of $K$ (containing $p$)}. In a similar way, we can define $k$-dimensional faces for $2 \le k \le n-2$.
\end{definition}

Note that the endpoints of an edge (if they exist) are vertices.

\begin{definition}[$\lambda$-convex lens]	\label{Def:LambdaLens}
A \emph{$\lambda$-convex lens} (or a \emph{UFO body}) is a $\lambda$-convex body given by the intersection of two sets of the form $\Bb_\lambda$ (see Figure~\ref{Fig:Lens}).
\end{definition}

Note that in our definition, a $\lambda$-convex lens can be a ball (if the two generating sets $\Bb_\lambda$ are balls that coincide). Otherwise, a $\lambda$-convex lens is a $\lambda$-convex polytope that has exactly two facets. Furthermore, for $n\ge3$, a $\lambda$-convex lens has exactly one $(n-2)$-dimensional face, and no faces of dimension less than $n-2$. For $n=2$, there are exactly two $0$-dimensional faces (i.e., the vertices of the $2$-dimensional lens). Each lens is reflection-symmetric with respect to the totally geodesic hyperplane passing through its $(n-2)$-dimensional face(s). 

In $\R^n$, a $\lambda$-convex lens is the intersection of two balls of radius $1/\lambda$.

We would like to point out a direct correspondence between $\lambda$-convex polytopes and \emph{ball-polyhedra}, i.e., bodies that are obtained as intersections of finitely many geodesic balls of a fixed radius. These objects from discrete geometry have been intensively studied; see e.g., \cite{BLNP} and references therein. The notion of $\lambda$-convexity in the context of ball-polyhedra appears in a number of papers as ``$r$-hyperconvexity", ``spindle convexity", and ``ball convexity", e.g., \cite{BLNP, Bezdek2018, FKV, LNT2013} (see also \cite{CLN} where a discrete isoperimetric inequality, in the classical direction, is discussed for spindle convex curves in the model planes). However, in our definition of $\lambda$-convexity in the hyperbolic space, we allow for supporting hypersurfaces to be not only spheres, but also horospheres or equidistants. The former case is known as an instance of \emph{$h$-convexity}. From this point of view, motivated by some questions in geometric probability, $\lambda$-convexity was studied in Riemannian spaces, e.g., in \cite{BGR, BM}.

The \emph{Hausdorff distance} between convex bodies $A, A'$ in $\R^n$ is defined by
\[
d_H(A,A') = \min\{\mu \ge 0: \ A \subseteq A' + \mu \B^n, \ A' \subseteq A + \mu \B^n \},
\] 
where $\B^n$ is the unit Euclidean ball in $\mathbb{R}^n$. The following proposition is a direct corollary of Blaschke's rolling theorem (Theorem~\ref{Thm:Bla}):

\begin{proposition}[Approximation by $\lambda$-convex polytopes]
	\label{Prop:Approx}
	Let $K \subset M^n(c)$ be a $\lambda$-convex body. Then there exists a sequence $(P_i)_{i=1}^\infty$ of $\lambda$-convex polytopes such that $K \subseteq \ldots \subseteq P_2 \subseteq P_1$ and $\lim\limits_{i \rightarrow\infty} P_i = K$ in the Hausdorff metric.
\end{proposition}

Observe that an implicit assumption in Definition~\ref{Def:LambdaConvexPol}  is that the intersection of balls in question is compact. In many cases, this is automatic:
\begin{lemma}[Compact intersection of balls]
	\label{Lem:Cpt1}
	Let $\Bb_1, \Bb_2 \subset M^n(c)$ be a pair of closed convex sets bounded by totally umbilical hypersurfaces of curvature $\lambda > 0$. Assume $\Bb_1$ is not contained in $\Bb_2$, and $\Bb_2$ is not contained in $\Bb_1$.
	If $c \ge 0$, or $c < 0$ and $\lambda \ge \sqrt{-c}$, then $\Bb_1 \cap \Bb_2$ is a compact set. 
\end{lemma}

\begin{proof}
	It follows from the classification of totally umbilical hypersurfaces, see Section~\ref{SSec:Facts}.
\end{proof}

 However,  Lemma~\ref{Lem:Cpt1} is false in this generality in the hyperbolic equidistant case, i.e., when $\lambda < \sqrt{-c}$. Since each $\Bb_i$ is unbounded and contains a cone of directions to infinity, this causes extra difficulty. Nonetheless, we can still get a compact intersection provided that sufficiently many equidistant balls are taken and that they are sufficiently `close' to each other. Recall definition~\eqref{Eq:rr} of the characteristic distance $\rr_\lambda$. More precisely, we have:

\begin{lemma}[Compact intersection of equidistant balls]
	\label{Lem:Cpt2}
	Let $B \subset \Hh^n(c)$, $c <0$, be a geodesic ball with radius $r$, and let $\Bb_1, \ldots, \Bb_{m}$ be a collection of $m$ convex regions in $\Hh^n(c)$ bounded by equidistants of curvature $\lambda \in (0, \sqrt{-c})$. Assume also that $B \subset \Bb_i$ and $\partial B$ touches $\partial \Bb_i$ at some point $p_i$  for each $i \in \{1, \ldots, m\}$. Suppose that 
    $r < \rr_\lambda$ and that the set of touching points ${\bf p} = \{p_1, \ldots, p_{m}\}$ does not lie in any open half-space with respect to any hyperplane passing through the center of $B$. Then 
	\[
	\Bb_{\bf p} := \bigcap_{i = 1}^{m} \Bb_i 
	\] 
	is a compact set.
\end{lemma}

\begin{remark}
	Lemma~\ref{Lem:Cpt2} is false if $r \ge \rr_\lambda$ (compare to Lemma~\ref{Lem:Unique} \eqref{It:3} below).
\end{remark}

\begin{proof}[Proof of Lemma~\ref{Lem:Cpt2}]
	For the proof, we need an auxiliary construction as follows. Let $\mathcal H \subset \Hh^n(c)$ be a hyperplane passing through the center of $B$. Denote by $S = S(\mathcal H) := \partial B \cap \mathcal H$ the corresponding great sphere. For $p \in S$, let $\Bb_\lambda(p) \supset B$ be the convex region bounded by the equidistant of curvature $\lambda$ that touches the ball $B$ at $p$. 
    The assumption $r < \rr_\lambda$ is equivalent to the fact that
	\[
	\mathcal Y := \bigcap_{p \in S} \Bb_\lambda(p)
	\]
	is a compact (convex) body. Equivalently, this means that for each such great sphere $S$, one can find a point, say $q = q(S) \in \Hh^n(c)$, such that each equidistant arc of curvature $\lambda$ joining $q$ with a point $q' \in S$ is perpendicular to $\mathcal H$ at $q'$. By symmetry, the point $q$ lies on the geodesic line $\gamma$ passing through the center of $B$ perpendicular to $\mathcal H$. Moreover, there are exactly two such points, one per each half-space with respect to $\mathcal H$. Call these points $q^+=q^+(S)$ and $q^-=q^-(S)$. Now, for a point $p \in \Hh^n(c)$, define $\mathcal X(B, p)$ to be the intersection of all convex sets of the form $\Bb_\lambda$ such that $B \subset \Bb_\lambda$ and $p \in \partial \Bb_\lambda$. In this notation,
	\[
	\mathcal X(B, q^+) = \mathcal X(B, q^-) = \mathcal Y.
	\]
	In general, $\mathcal X(B,p)$ is not necessarily compact. However, since $\mathcal X(B, q^\pm)$ are compact, for every point $p$ sufficiently close in the hyperbolic metric to either $q^+$ or $q^-$, the set $\mathcal X(B, p)$ is also compact. Moreover, if additionally $p$ lies on the geodesic $\gamma$ outside of the segment $[q^-,q^+]$ and closer to, say $q^+$, then both $\partial \mathcal X(B,p) \cap \partial B$ and $q^-$ are contained in the same open half-space with respect to $\mathcal H$. By symmetry, the intersection $\partial \mathcal X(B,p) \cap \partial B$ lies in some hyperplane $\mathcal I(p)$. Let $\mathcal I^+(p)$ be the open half-space with respect to $\mathcal I(p)$ that contains $q^+$. Finally, set
	\[
	\mathcal Y^+(p) := B \cup \left(\mathcal X(B,p) \cap \mathcal I^+(p)\right).
	\]  
	
	Let us return to the proof of the lemma. Assume the contrary. Since $\Bb_{\bf p}$ is convex but not compact, it must contain a geodesic ray starting at an arbitrary point inside. Let $\gamma \subset \Bb_{\bf p}$ be such a geodesic ray that starts at the center of $B$ and goes to infinity. Let $\mathcal H_\gamma$ be the hyperplane passing through the center of $B$ perpendicular to $\gamma$, and let $S_\gamma := S(\mathcal H_\gamma)$ be the corresponding great sphere. Since $B \subset \Bb_{\bf p}$, Blaschke's rolling theorem (Theorem~\ref{Thm:Bla}) guarantees that each equidistant arc of curvature $\lambda$ that joins a point in $\gamma$ with a point in $\partial B$, lies in $\Bb_{\bf p}$. In particular, $\mathcal Y^+(p) \subset \Bb_{\bf p}$ for some $p \in \gamma$ sufficiently close to $q^+(S_\gamma) \in \gamma$. But $S_\gamma$ lies in the interior of $\mathcal Y^+(p)$ by construction. This is a contradiction since the open half-space with respect to $\mathcal H_\gamma$ that contains $p$ does not contain points in ${\bf p}$.
\end{proof}

\subsection{Inscribed balls and their properties}

Let $K \subset M^n(c)$ be a convex body. A geodesic ball of the largest possible radius contained in $K$ is called an \emph{inscribed ball for $K$}. The \emph{inradius of $K$}, denoted by $r(K)$, is the radius of an inscribed ball for $K$. Note that $r(K)$ is well-defined even though a convex body can have several inscribed balls. The question of uniqueness is addressed in the following lemma: 

\begin{lemma}[On uniqueness of inscribed balls]
\label{Lem:Unique}
Let $K \subset M^n(c)$ be a strictly convex body. 
\begin{enumerate}
\item
\label{It:1}
If $c \ge 0$, then $K$ has a unique inscribed ball;
\item
\label{It:2}
If $c < 0$ and $K$ is $\lambda$-convex with $\lambda \ge \sqrt{-c}$, then $K$ has a unique inscribed ball;
\item
\label{It:3}
For each $c < 0$ and $\lambda' \in (0, \sqrt{-c})$, there is a $\lambda'$-convex body $K$ such that $K$ has exactly two inscribed balls; in particular, in the hyperbolic space, a $\lambda$-convex body can have several inscribed balls.
\end{enumerate}
\end{lemma}

\begin{proof}
The argument in the Euclidean space is well-known and goes, for example, as follows. If there are two distinct balls $B_1, B_2 \subset K \subset \R^n$, necessarily of the same radius $r:=r(K)$, then the convex hull $H$ of $B_1$ and $B_2$ lies entirely in $K$. Moreover, since $K$ is strictly convex, any ball $B \subset H$ of radius $r$ different from $B_1$ and $B_2$ must lie in the {interior} of $K$. Hence, one can enlarge $B$ (keeping its center the same) so that the larger ball still lies in $K$. This contradicts the maximality of  $r$. 

We can reinterpret this argument so that it also works in $\S^n$. Again, by contradiction, we assume that there are two different inscribed balls $B_1$ and $B_2$ for $K \subset \S^n$.

Let $H'$ be the union of balls of radius $r$ centered on the geodesic segment $I$ connecting the centers of $B_1$ and $B_2$. Equivalently, $H'$ is the set of points at distance at most $r$ from $I$. In Euclidean space, $H' = H$. In the spherical case, we have that $H' \subset H \subset K$. Indeed, the first inclusion follows from the fact that $H'$ is not geodesically convex, and hence is contained in the convex hull $H$; using symmetry, one can reduce this question to dimension $2$, where the claim can be verified by a direct computation. Applying an argument similar to the Euclidean case above to $H'$ yields that a strictly convex body $K \subset \S^n$ has a unique inscribed ball. This concludes~\eqref{It:1}.

Now we consider the hyperbolic case. For simplicity, assume $c = -1$ and write $\Hh^n := \Hh^n(-1)$. The inclusion $H' \subset H$ does not hold in $\Hh^n$. Observe, however, that $\partial H' \setminus (\partial B_1 \cup \partial B_2)$ is a smooth hypersurface with principal curvatures bounded below by $\mu = \tanh (r) \in (0,1)$ (i.e., $r = \rr_\mu$, see \eqref{It:Horo} in Section~\ref{SSec:Facts}). If we are in the assumption of case \eqref{It:2}, then $\mu < \lambda$. By Blaschke's rolling theorem (see the discussion after Theorem~\ref{Thm:Bla}), it follows that every arc of curvature $\mu$ is contained in $K$, and thus $H' \subset K$. From here, we can reach a contradiction in the same way as in $\R^n$ and $\S^n$. This concludes \eqref{It:2}.	

\begin{figure}
	\includegraphics[trim= 20 20 20 20, clip, scale=1.7]{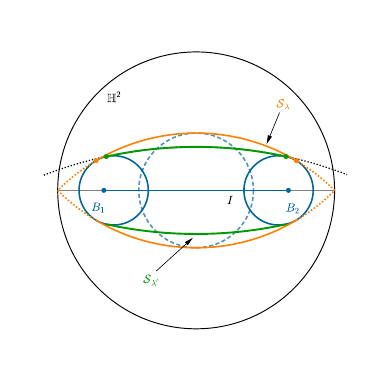}
	\caption{In the Poincar\'e disk model of $\Hh^2$: the body $H''$ is the concatenation of two circular arcs on the boundaries of disks $B_1$ and $B_2$, and two arcs of curves $\Ss_{\lambda'}$ of constant curvature $\lambda' \in (0,1)$. The outer parallel set $H'$ of the segment $I$ is the concatenation of two arcs $\Ss_\lambda$ and two circular arcs on the boundaries of the disks. For comparison, the disk centered at the origin has the same radius as $B_1$ and $B_2$. In this example, $B_1$ and $B_2$ are the only inscribed disks for $H''$.}
	\label{Fig:Example}
\end{figure}

In general, there is no guarantee that $H' \subset K$. Also, $H'$ provides an example of a $\mu$-convex body, with $\mu \in (0,1)$, for which all balls of radius $r$ inside $H'$ are inscribed, and there is a continuous family of such balls. To conclude \eqref{It:3}, one can modify this example as follows (see Figure~\ref{Fig:Example} for a construction in $\Hh^2$). Given $\lambda' \in (0,1)$, pick a geodesic segment $I \subset \Hh^n$ and let $H'$ be defined as above for $r > \rr_{\lambda'}$. Consider all possible arcs of curves of curvature $\lambda'$ joining points in $B_1$ and $B_2$. These arcs will sweep out a $\lambda'$-convex body $H''$ for which $B_1$ and $B_2$ are the only inscribed balls.
\end{proof}

The following proposition is folklore in some settings. We state it for $\lambda$-convex bodies in model spaces; however, for spherical and Euclidean spaces it also holds under some weaker assumptions. On the other hand, in the hyperbolic setting, the version that we state is sharp.

\begin{proposition}[Characterization of inscribed balls]
\label{Prop:Char}
Let $K \subset M^n(c)$ be a $\lambda$-convex body and let $B \subseteq K$ be a ball of radius $r$.

If $B$  is an inscribed ball of $K$, then $\partial B \cap \partial K$ is not contained in any open hemisphere of $\partial B$.
Conversely, if  $\partial B \cap \partial K$ is not contained in any open hemisphere of $\partial B$, and in the case $c < 0$ and $\lambda < \sqrt{-c}$, we additionally assume that $r < \rr_\lambda$, then $B$  is an inscribed ball of $K$.
\end{proposition}

\begin{proof}

$(\Rightarrow)$ Let $B \subseteq K$ be an inscribed ball centered at $o \in K$ of radius $r$. Towards a contradiction, suppose that $\partial B \cap \partial K$ lies in some open half-space $\mathcal H^+$. Write $\mathcal H := \partial \mathcal H^+$, $o \in \mathcal H$. Let $\mathcal H^- := M^n(c) \setminus \mathcal H^+$ be the complement of $\mathcal H^+$. Pick a small $\eps > 0$, and consider the geodesic segment  $I_\eps \subset \mathcal H^-$ of length $\eps$, starting at $o$ and perpendicular to $\mathcal H$. Define $H'_\eps$ to be the union of balls of radius $r$  centered at $I_\eps$. Since $B$ and $K$ are compact and $\partial B \cap \partial K$ is contained in $\mathcal H^+$, the distance
\[
\dist_{M^n(c)} \left(\partial K, \mathcal H^- \cap B\right)
\] 
is strictly positive. Therefore, there exists $\eps$ such that $H'_\eps \subset K$. To reach a contradiction,  we can apply an argument similar to the one that was used in the proof of Lemma~\ref{Lem:Unique}~\eqref{It:1}.

$(\Leftarrow)$ We first prove the implication in $\mathbb R^n$, and then adapt the proof to the other constant curvature spaces.

Let $B, B' \subseteq K$ be a pair of balls centered at $o, o' \in K$ and of radii $r, r' > 0$. For the proof, we assume that both sets $\partial B \cap \partial K$ and $\partial B' \cap \partial K$ are not empty (otherwise, we can always increase the radii while keeping the centers fixed). 

\textbf{Case $c=0$. }
Let $x \in \partial B \cap \partial K$. The vector $(x-o)/r$ is the unit outward-facing normal to $\partial B$ as well as to $\partial K$. Since $K$ is convex, it lies in a closed half-space $V_x$ with respect to the supporting hyperplane passing through $x$ with normal $(x-o)/r$ so that 
\[
V_x := 
\left\{y \in \mathbb R^n : \left\langle y - x,  \frac{x-o}{r}\right\rangle \le 0\right\} = 
\left\{y \in \mathbb R^n : \left\langle y - o,  \frac{x-o}{r}\right\rangle \le r\right\}
\]
where $\langle \cdot, \cdot \rangle$ is the standard scalar product in $\R^{n}$.
In particular, the point $o'+\frac{r'}{r}(x-o) \in \partial B' \subset V_x$. Hence,
\[
\left\langle o' - o +\frac{r'}{r}(x-o),  \frac{x-o}{r}\right\rangle \le r. 
\]
Simplifying, we obtain
\begin{equation}
\label{Eq:Half1}
\left\langle o' - o,  \frac{x-o}{r}\right\rangle \le r - r', \quad \forall x \in \partial B \cap \partial K.
\end{equation}

If $\partial B \cap \partial K$ is not contained in any open hemisphere, then there exists $x \in \partial B \cap \partial K$ such that the left-hand side in \eqref{Eq:Half1} is non-negative. Therefore, $r \ge r'$ and $B$ is the ball of the largest radius contained in $K$. 

\textbf{Case $c > 0$.}  Without loss of generality, we can assume that $c = 1$. We write $\S^n = \S^n(1)$. Fix $o \in \S^n$ as the origin of the polar coordinate system. Each point $x \in \S^n \subset \R^{n+1}$ is given by the vector
\[
x = \cos a \cdot o + \sin a \cdot v, \qquad   v \in \S^{n} \cap o^{\perp}, \quad a = \dist_{\S^n}(o, x),
\] 
where $o$ is the unit vector in $\R^{n+1}$ and $v$ is a unit vector in the orthogonal complement $o^\perp$ to $o$. In particular, for $x\in \partial B$ we have $a=r$, and hence
\[
x=\cos r \cdot o+\sin r \cdot v, \qquad v\in \S^n\cap o^\perp .
\]
The unit outward normal to $\partial B$ at $x$ is
\[
\nu_x=-\sin r \cdot o+\cos r \cdot v \in T_x\S^n 
\]
where $T_x\S^n$ is the tangent space to $\S^n$ at $x$.
The totally geodesic hyperplane tangent to $\partial B$ at $x$ and orthogonal to $\nu_x$ is
\[
P_x:=\{y\in \S^n:\langle y,\nu_x\rangle=0\},
\]
where $\langle \cdot, \cdot \rangle$ stands for the scalar product in $\R^{n+1}$. Since $K$ is $\lambda$-convex, it is, in particular, geodesically convex. Therefore, $P_x$ is a supporting totally geodesic hyperplane to $K$ at $x$ and
\[
K \subset \left\{y\in \S^n:\langle y,\nu_x\rangle \le 0\right\}.
\]
Let
\[
o'=\cos d\cdot o+\sin d\cdot w,
\qquad 
w\in \S^n\cap o^\perp,
\qquad 
d=\dist_{\S^n}(o,o').
\]
Since $B'\subseteq K$, the ball $B'$ is contained in the same hemisphere as $K$ with respect to $P_x$. Equivalently, the spherical distance from $o'$ to $P_x$ on the side containing $K$ is at least $r'$. In terms of the signed distance to $P_x$, this gives
\[
\langle o',\nu_x\rangle \le -\sin r' .
\]
Using the expressions for $o'$ and $\nu_x$, we obtain
\[
\langle o',\nu_x\rangle =-\sin r\cos d+\cos r\sin d\,\langle w,v\rangle .
\]
Therefore,
\begin{equation}
\label{Eq:SphHalf}
\cos r\sin d\,\langle w,v\rangle \le \sin r\cos d-\sin r' \le \sin r - \sin r' , \qquad \forall x \in \partial B\cap \partial K.
\end{equation}
 We note that this estimate is analogous to \eqref{Eq:Half1}. 

If $\partial B\cap \partial K$ is not contained in any open hemisphere of $\partial B$, then for every direction $w \in \S^n\cap o^\perp$ there exists $v \in \S^n\cap o^\perp$ (corresponding to the point $x=\cos r \cdot o+\sin r \cdot v \in \partial B \cap \partial K$) such that $\langle v, w \rangle \ge 0$. By \eqref{Eq:SphHalf}, this implies $\sin r \ge \sin r'$, and hence $r \ge r'$ (since both of these radii are at most $\pi/2$). This concludes the proof in the spherical case.

\textbf{Case $c<0$.} We assume $c=-1$ and write $\mathbb H^n = \mathbb H^n(-1)$. We use the hyperboloid model of the hyperbolic space, that is,
\[
\mathbb H^n=\{x = (x_0, x_1, \ldots, x_n)\in \R^{n,1}:\langle x,x\rangle_{n,1} = -x_0^2 + x_1^2 + \ldots + x_n^2=-1,\ x_0>0\}.
\]

We can further assume that $K$ is not a geodesic ball (otherwise, the claim of the proposition is trivial). Choose $\mu < 1$ so that 
\begin{equation}
\label{Eq:AssMu}
\text{$K$ is $\mu$-convex and $r < \rho$, where $\rho := \tanh^{-1}(\mu)$.}
\end{equation}
This choice is made as follows. If $\lambda < 1$, then we set $\mu = \lambda$, thus $\rho = \rr_\lambda$, and $r < \rho$ follows from the assumption of the proposition. 
If $\lambda \ge 1$, choose any $\mu \in (\tanh r, 1)$. Then $r < \tanh^{-1}(\mu) = \rho$ and $K$ is $\mu$-convex since $\mu < 1 \leqslant \lambda.$

Similar to the case of spherical space, every point $x\in \mathbb H^n$ can be written in polar coordinates as
\[
x=\cosh a\cdot o+\sinh a\cdot v, \qquad a=\dist_{\mathbb H^n}(o,x),
\]
where $o \in \R^{n,1}$ is a unit time-like vector and  $v \in \R^{n,1}$ is a unit space-like vector orthogonal to $o$ in the pseudo-metric $\langle \cdot, \cdot \rangle_{n,1}$. Thus, for $x\in \partial B$, we have
\[
x=\cosh r\cdot o+\sinh r\cdot v .
\]
The unit outward normal to $\partial B$ at $x$ is
\[
\nu_x=\sinh r\cdot o+\cosh r\cdot v \in T_x \mathbb H^n
\]
where $T_x \mathbb H^n$ is the tangent space to $\mathbb H^n$ at $x$.

By $\mu$-convexity of $K$, there exists a supporting equidistant hypersurface of curvature $\mu$ passing through $x$ with outward normal $\nu_x$ such that $K$ lies on its inward side. Since $\mu = \tanh \rho$, it is known that this hypersurface has the form
\[
S_x := \{y \in \mathbb H^n : \langle y, N_x \rangle_{n,1} = \sinh \rho\}
\]
given by a hyperplane in $\R^{n,1}$ parallel to $\{y \in \mathbb H^n : \langle y, N_x \rangle_{n,1} = 0\}$ at hyperbolic distance $\rho$. Here, $N_x \in T_x \R^{n,1}$ is the unit space-like vector normal to $S_x$ at $x$. A direct computation gives
\[
N_x = \cosh \rho \cdot \nu_x - \sinh \rho \cdot x = -\sinh(\rho-r) \cdot o + \cosh(\rho-r) \cdot v.
\]  

As before, we write
\[
o'=\cosh d\cdot o+\sinh d\cdot w, \qquad d=\dist_{\mathbb H^n}(o,o').
\]
The condition $B'\subset K \subset \{y \in \mathbb H^n : \langle y, N_x \rangle_{n,1} \le \sinh \rho\}$ implies that $o'$ lies in the inner parallel set at distance $r'$  for $\{y \in \mathbb H^n : \langle y, N_x \rangle_{n,1} \le \sinh \rho\}$, i.e.,
\[
\langle o', N_x\rangle_{n,1}\le \sinh (\rho - r').
\]
Thus,
\[
\langle o',N_x\rangle_{n,1} = \cosh d\sinh (\rho-r) + \sinh d \cosh(\rho - r)\,\langle w,v\rangle_{n,1} \le  \sinh (\rho - r'),
\]
which, using $r < \rho$ and $\cosh d \ge 1$, implies
\begin{equation}
\label{Eq:HypHalf}
\sinh d \cosh(\rho - r)\,\langle w,v\rangle_{n,1} \le \sinh (\rho - r') - \sinh (\rho - r).
\end{equation}
Note that $\langle w,v\rangle_{n,1} = \langle w,v\rangle$, where the latter scalar product can be viewed in $o^\perp \equiv \R^{n}$. Using \eqref{Eq:HypHalf}, the conclusion follows exactly as in the spherical case above.
\end{proof}

The following proposition relates the extreme properties of the inradius to the structure of the boundary of a $\lambda$-convex body. 

\begin{proposition}[Extreme properties of the inradius]
\label{Prop:Reduction}
Let $K \subset M^n(c)$ be a $\lambda$-convex body. (In the case $c < 0$ and $\lambda < \sqrt{-c}$, we additionally assume that $r(K) < \rr_\lambda$.)

If $K \subset M^n(c)$ is not a $\lambda$-convex polytope, then one can find a $\lambda$-convex polytope $K'$ such that 
\begin{equation}
\label{Eq:Red}
r(K') = r(K) \quad \text{ and }\quad |\partial K'| > |\partial K|.
\end{equation}
Equivalently, there exists a $\lambda$-convex polytope $K''$ such that
\begin{equation}\label{Eq:Red2}
r(K'') < r(K) \quad \text{ and }\quad |\partial K''| = |\partial K|.
\end{equation}
\end{proposition}

\begin{proof}
We start with \eqref{Eq:Red}. 

Let $B \subseteq K$ be an inscribed ball of radius $r = r(K)$ centered at $o$. Furthermore, let $T_o M^n(c)$ be the tangent space to $M^n(c)$ at $o$; we identify $(T_o M^n(c), 0)$ with $(\R^n, 0)$. Finally, let $\mathbf{B}^n_r \subset \R^n \simeq T_o M^n(c)$ be the ball of radius $r$ centered at the origin. The exponential map $\exp_o \colon T_o M^n(c) \to M^n(c)$ sends $\mathbf{B}^n_r$ diffeomorphically onto $B$.

By Proposition~\ref{Prop:Char}, $\partial B \cap \partial K$ is not contained in any open hemisphere of $\partial B$. Equivalently, the set $S := \exp_o^{-1}(\partial B \cap \partial K)$ is not contained in any open hemisphere of $\partial \mathbf{B}^n_r$. The latter property is equivalent to the fact that the convex hull of $S$ in $\R^n$ contains $0$. By the Carath\'eodory theorem applied to $S$ in $T_o M^n(c) \simeq \R^n$, we can find a subset $S' \subset S$ consisting of at most $n+1$ points such that the convex hull of $S'$ in $\R^n$ also contains $0$. Hence, the set $T := \exp_o S' \subset \partial B$ has the property that it is not contained in any open hemisphere of $\partial B$.

Let us now go back to the proof of the proposition. Since $K$ is $\lambda$-convex, for each $p \in T$ there exists a supporting convex set $\Bb_\lambda(p)$ of curvature $\lambda$ passing through $p$. Define
\[
K' := \bigcap_{p \in T} \Bb_\lambda(p).
\] 
By Lemmas~\ref{Lem:Cpt1} and \ref{Lem:Cpt2}, $K'$ is a $\lambda$-convex polytope. By Blaschke's rolling theorem (Theorem~\ref{Thm:Bla}), $K' \supset K$. Hence, $|\partial K'| > |\partial K|$ (we recall that monotonicity of the surface area in spherical and hyperbolic space follows directly from Crofton's formula, see for example \cite[eq. (17.61)]{Sa}). This inequality is strict because $K$ is not a $\lambda$-convex polytope by assumption. Furthermore, $B \subset K'$ and $T = \partial B \cap \partial K'$ by construction. By Proposition~\ref{Prop:Char}, $B$ is an inscribed ball for $K'$. Therefore, $r(K') = r(K)$. This concludes \eqref{Eq:Red}.

To prove \eqref{Eq:Red2}, we use \eqref{Eq:Red} and a monotonicity argument. Indeed, let $K'$ be the polytope that satisfies \eqref{Eq:Red}, let $B$ be its inscribed ball of radius $r = r(K)$, and $T = \partial B \cap \partial K'$. For $s \le r$, let $B'$ be a ball of radius $s$ concentric with $B$. For each $p\in T$, let $p'$ be the intersection of the geodesic ray $op$ with $\partial B'$. Let $T'$ be the set of all such points $p'$. For each $p'$, let $\Ss_\lambda(p')$ be the $\lambda$-convex hypersurface that touches $B'$ at $p'$ in such a way that $\Bb_\lambda(p') \supset B'$. Finally, define 
\[
K''(s) := \bigcap_{p' \in T'} \Bb_\lambda(p').
\] 
Since $s \le r$, $K''(s)$ is a $\lambda$-convex polytope (in the case $c < 0$ and $\lambda < \sqrt{-c}$, compactness follows from Lemma 2.7 since $T'$ does not lie in any open half-space through $o$ and $s < \rr_\lambda$). By Proposition~\ref{Prop:Char}, $r\left(K''(s)\right) = s$. By construction, $K''(r) = K'$, and hence, $|\partial K''(r)| > |\partial K|$. Finally, since $T'$ does not lie in any open half-space through $o$, $K''(s) \rightarrow \{o\}$ as $s \to 0$. Thus, $|\partial K''(s)|$ tends to $0$ as $s \to 0$. Also, $K''(s)$ depends continuously on $s$ in Hausdorff metric since the finitely many sets $\Bb_\lambda(p'(s))$ depend continuously on $s$, so $|\partial K''(s)|$ is continuous in $s$. Therefore, by continuity, there exists $s_0 < r$ such that $|\partial K''(s_0)| = |\partial K|$, and the proposition follows with $K'' = K''(s_0)$. 
\end{proof}

\begin{remark}
We will use a similar monotonicity argument in the proof of Theorem~\ref{Thm:MainB}.
\end{remark}

\subsection{Inner parallel bodies in $\R^n$ and their properties}

Let $A$ be a convex body in $\R^n$. The \emph{inner parallel body} at distance $t \ge 0$ is defined as
\[
A_t = A \sim t \B^n := \{x \in \R^n : x + t \B^n \subseteq A\}.
\] 
The greatest number $t$ for which $A_t$ is not empty is the inradius $r(A)$ of $A$. Observe that if $A \subseteq A' \subset \R^n$ are two convex bodies, then $A_t \subseteq A'_t$.

\begin{lemma}[Dependence on $\eps$]
\label{Lem:Mon1}
\[
d_H(A, A_\eps) \to 0 \quad \text{as} \quad \eps\searrow 0. 
\]
\end{lemma}

\begin{proof}
For monotonically decreasing $\eps$, the bodies in the sequence $A_\eps = A \sim \eps \B^n$ are nested, i.e., $A_\eps \subset A_{\eps'}$ if $\eps > \eps'$. Also, $A_{\varepsilon} \subset A$ and for any $p \in \text{int}\, A$, we have $p \in A_\eps$ for any $0 < \eps < \dist({p, \partial A})$. Hence, $\text{int}\, A \subseteq \bigcup_{\eps > 0} A_\eps$.  Thus, $A$ is equal to the closure of $\bigcup_{\eps > 0} A_\eps$, and the claim follows from \cite[Theorem 1.8.8]{Sch}.

\end{proof}

\begin{lemma}[Dependence on $A$]
\label{Lem:Mon2}
Let $A$ and $A'$ be convex bodies in $\R^n$ such that $A' \supseteq A$. Then 
\[
d_H(A', A'_\eps) \le d_H(A, A_\eps) + d_H(A,A')
\]
 for any $0 \le \varepsilon \le r(A)$.
\end{lemma}

\begin{proof}
Since $A \subseteq A'$ then $A_{\varepsilon} \subseteq A'_{\varepsilon} \subseteq A'$. Hence, 
\begin{align*}
    d_H(A', A'_{\varepsilon}) \le d_H(A', A_{\varepsilon}) \le d_H(A', A) + d_H(A, A_{\varepsilon})
\end{align*}
where we used the triangle inequality.
\end{proof}

The next lemma follows from \cite[Lemma~1.8.2]{Sch} and \cite[Lemma~1.8.6]{Sch}.

\begin{lemma}[Hausdorff limits of nested sequences]
\label{Lem:HD3}
If $A^1 \supseteq A^2 \supseteq ... \supseteq C$ is a nested sequence of convex bodies, then 
\[
\mathcal A := \bigcap_{j=1}^\infty A^j
\]
is a convex body and 
\[
d_H\left(A^j, \mathcal A\right) \to 0 \quad \text{ as }\quad j \to \infty.
\]
\end{lemma}

Using the three lemmas above, we establish the following auxiliary result. We want to remark that Lemma~\ref{Lem:Des} can also be proved in $\R^n$ using support functions and Wulff shapes, see \cite[Lemma 7.5.2]{Sch}.

\begin{lemma}[Approximation descends to inner parallel bodies]
\label{Lem:Des}
If $A \subset \R^n$ is a convex body, and $(A^j)_{j=1}^\infty$ is a sequence of convex bodies such that 
\[
A^j \supseteq A^{j+1} \supseteq A  \quad \text{ and }\quad \lim_{j \to \infty} d_H\left(A, A^j \right) = 0, 
\]
then for the inner parallel bodies at distance $0 \le t < r(A)$, we have 
\[
\lim_{j \to \infty} d_H\left(A_t, A_t^j \right) = 0.
\]
\end{lemma}

\begin{proof}

Since $(A^j)_{j=1}^\infty$ is a nested sequence that approximates $A$, there is a monotone sequence $\eps_j \searrow 0$ such that 
\[
d_H(A, A^j) = \eps_j.
\]
Hence,
\[
A + \eps_j \B^n \supseteq A^j \ \  \Rightarrow \ \ A   \supseteq A^j \sim \eps_j \B^n \ \  \Rightarrow  \  \ A \sim t \B^n \supseteq A^j \sim \eps_j \B^n \sim t \B^n.
\]
In our notation, the last inclusion also reads as $A_t \supseteq A_t^j \sim \eps_j \B^n$. Since
\[
A_t^j \sim \eps_j \B^n + d_H(A_t^j, (A_t^j)_{\eps_j}) \B^n \supseteq A_t^j,
\]
we conclude that
\begin{equation}
\label{Eq:HD1}
A_t + d_H(A_t^j, (A_t^j)_{\eps_j}) \B^n \supseteq A_t^j.
\end{equation}

By assumption of the lemma, $A^j \supseteq A$, so  $A^j_t \supseteq A_t$. Hence, \eqref{Eq:HD1} implies that
\begin{equation*}
d_H\left(A_t, A_t^j\right) \le d_H(A_t^j, (A_t^j)_{\eps_j}) .
\end{equation*}
It remains to show that $d_H(A_t^j, (A_t^j)_{\eps_j}) \to 0$ as $j \to \infty$.

Since the sequence $\left(A^j\right)_{j=1}^\infty$ is decreasing, so is the sequence $\left(A^j_t\right)_{j=1}^\infty$. Moreover, each $A^j_t$ contains $A_t$. The latter set is a (non-empty) convex body by our choice of $0 \le t < r(A)$. Hence, by Lemma~\ref{Lem:HD3},
\[
\mathcal A := \bigcap_{j = 1}^\infty A_t^j
\] 
is a convex body and is the Hausdorff limit of $A_t^j$ as $j \rightarrow \infty$. Since $A_t^j \supseteq \mathcal A$, Lemma~\ref{Lem:Mon2} guarantees that
\[
d_H(A_t^j, (A_t^j)_{\eps_j}) \le d_H(\mathcal A, \mathcal A_{\eps_j})   + d_H(\mathcal A, A_t^j)
\]
 for each $j \ge 1$.
By Lemma~\ref{Lem:HD3},
\[
d_H(\mathcal A, A_t^j) \to 0 \quad \text{ as }\quad j \to \infty.
\]
Since $\eps_j \searrow 0$ as $j \to \infty$, Lemma~\ref{Lem:Mon1} guarantees that
\[
d_H(\mathcal A, \mathcal A_{\eps_j}) \to 0 \quad \text{ as }\quad j \to \infty.
\]
Thus, $\lim\limits_{j \to \infty} d_H(A_t^j, (A_t^j)_{\eps_j})  = 0$ which concludes the proof.
\end{proof}

\section{Proof of the reverse isoperimetric inequality in $\R^3$ (Theorem~\ref{mainthm})}
\label{Sec:ProofA}

The proof breaks into two steps. In the first step (Subsection~\ref{SSec:StepI}), we establish the theorem when $K$ is a $\lambda$-convex polytope. In the second step (Subsection~\ref{SSec:StepII}), we use the approximation (Proposition~\ref{Prop:Approx}) and the result of the first step to prove the theorem in general: the reverse isoperimetric inequality will follow by continuity, and thus the only thing that we would have to deal with is the equality case.

Theorem~\ref{mainthm} is trivially true when $K$ is a ball of radius $1/\lambda$. Therefore, for the rest of this section we implicitly assume that $K$ is not such a ball, and hence the lens $L$ with $|\partial L| = |\partial K|$ cannot be a ball either.

\subsection{On deformations of convex domains}
\label{SSec:Deformation}

In this subsection, we discuss results on deformations of convex domains with piecewise smooth boundary in constant curvature spaces which will be used in Subsection~\ref{SSec:StepI}. The discussion in this subsection was inspired by and follows the work of Csik\'os~\cite{Cs} (see also \cite{So}).

Let $\Sigma \subset M^n(c)$ be a convex body with piecewise smooth boundary $\partial \Sigma$ and $n \ge 2$. We assume that $\Sigma = \bigcap_{i=1}^m \Sigma_i$, where each $\Sigma_i \subset M^n(c)$ is a convex body with smooth boundary. Now, we define the deformation of $\Sigma$ in the following way (compare to Definition~\ref{Def:LambdaConvexPol} and \cite[Definition 2.2]{Cs}). We note that this is not the most general definition that works, but it is enough for our purposes.

\begin{definition}[$2$-transversal variation of piecewise smooth convex bodies]
\label{Def:Def}
Let $I = [0, \eps)$ be a half-open interval. A \emph{variation} of a convex body $\Sigma = \bigcap_{i=1}^m \Sigma_i$ with piecewise smooth boundary is a one-parameter family of convex bodies $\Sigma(t) \subset M^n(c)$, $t \in I$, such that:
\begin{enumerate}
\item
There exists a compact set $\Omega \subset M^n(c)$ such that $\Sigma(t) \subset \Omega$ for all $t \in I$.
\item 
For every $i \in \{1, \ldots, m\}$ there exists a $C^1$-smooth map
\[
\Phi_i \colon \Sigma_i \times I \to M^n(c)
\]
such that $\Phi_i(\cdot,t)$ is an embedding with image $\Sigma_i(t)$ and $\Phi_i(p,0) = p$ for all $p \in \Sigma_i$; we assume that each $\Sigma_i(t)$ is a convex body.
\item 
For every $t \in I$,
\[
\Sigma (t) = \bigcap_{i=1}^m\Sigma_i(t);
\]
in particular, $\Sigma(0) = \Sigma$.
\item
({\it local $2$-transversality}) For every $t \in I$, if there exist $p \in \partial \Sigma(t)$ and $i \neq j$ such that $p \in \partial \Sigma_i(t) \cap \partial \Sigma_j(t)$, then the regular hypersurfaces $\partial\Sigma_i(t)$ and $\partial\Sigma_j(t)$ meet transversely at $p$. 
\end{enumerate} 
\end{definition}

For such a variation, we define 
\[
F_i(t) := \partial \Sigma_i(t) \cap \Sigma(t) = \partial \Sigma_i(t) \cap \partial \Sigma(t)
\]
and set $F_i := F_i(0)$. Each $F_i$ serves as a ``curved facet" of $\Sigma$, although in our definition it can have an empty relative interior in $\partial \Sigma$ while the interior of $F_i(t)$ for $t >0$ in $\partial \Sigma(t)$ is not empty. On the other hand, if some $F_i$ has the non-empty relative interior in $\partial \Sigma$, then for sufficiently small $\eps$, the same is true for $F_i(t)$. 

For $i \in \{1, \ldots, m\}$, along $F_i$ there is a well-defined vector field 
\[
X_i(p) := \frac{\partial\Phi_i}{\partial t}(p,0), \qquad p\in F_i = F_i(0).
\]
By the $2$-transversality, if $p \in \partial \Sigma$ is a smooth point, then there exists a unique index $i(p)$ such that $p \in F_{i(p)}$. A \emph{variational vector field along the smooth parts of $\partial \Sigma$ compatible with the variation $t \mapsto \Sigma(t)$} is a vector field $X$ defined at each smooth point $p \in \partial \Sigma$ by $X(p) = X_i(p)$, where $i = i(p)$. The vector field $X$ is defined everywhere except on a set of measure zero in $\partial \Sigma$. 

The following theorem follows from Stokes' theorem and is well-known in the smooth setting (see, e.g., \cite[Proposition 1.19]{Ri}). In the piecewise smooth setting, see e.g.\ \cite[Theorem 1]{Cs} (and also \cite[Lemma 1]{So}).

\begin{theorem}[Variation of volume]
\label{Thm:Variation}
Let $t \mapsto \Sigma(t) \subset M^n(c)$, $t \in [0, \eps)$, $n \ge 2$, be a $2$-transversal deformation of a convex body $\Sigma = \Sigma(0) \subset M^n(c)$ in the sense of Definition~\ref{Def:Def}. Let $N$ be the unit outer normal vector field to $\partial \Sigma$ and $X$ be the variational vector field compatible with the variation (both defined along the smooth part of $\partial \Sigma$); then
\begin{equation}
\label{Eq:Variation}
\left.\frac{d}{dt}\right|_{t=0^+}|\Sigma(t)| = \int_{\partial \Sigma} \left< X, N\right> \text{d} \sigma_{\Sigma},
\end{equation}
where $\left<\cdot,\cdot\right>$ is the Riemannian metric in $M^n(c)$ and $d \sigma_{\Sigma}$ is the area measure on $\partial \Sigma$.
\end{theorem}

Note that the non-smooth parts of $\partial \Sigma$ have zero area with respect to the measure $d \sigma_{\Sigma}$, and hence the integral \eqref{Eq:Variation} makes sense as the integral along the smooth parts of the boundary.  

In $\mathbb R^n$, inner parallel bodies $K_t$ at distance $t$ of a convex body $K$ provide an example of a variation of $K$. If $K$ has a piecewise smooth boundary, then the variational field compatible with the deformation by inner parallel bodies is equal to $-N$ in the interior of each smooth component of $\partial K$. Therefore, $\frac{d}{dt}|_{t=0^+} |K_t| = - |\partial K_0| = - |\partial K|$ by Theorem~\ref{Thm:Variation}. In the context of inner parallel bodies in $\mathbb R^n$, a stronger result holds.

\begin{theorem}[\cite{Ma}]
\label{Thm:Ma}
If $K_t$ is the inner parallel body at distance $t \ge 0$ of a convex body $K \subset \R^n$, $n \ge 2$, then 
\[
\frac{d}{dt} |K_t| = - |\partial K_t| \quad \text{ for all }t \in [0, r(K)).
\]
For $t=0$, the corresponding one-sided derivatives are considered. 
\end{theorem}

\subsection{Step I: $K$ is a $\lambda$-convex polytope}
\label{SSec:StepI}

We first prove Theorem~\ref{mainthm} in the case where $K$ is a $\lambda$-convex polytope in $\R^3$ (see Definition~\ref{Def:LambdaConvexPol}). Without loss of generality, we set $\lambda = 1$. We can also assume that $K$ is not a lens (as for lenses the claim of the theorem is true). By Theorem~\ref{Thm:Ma},
\begin{equation}
\label{Eq:IntMa}
|K| = \int \limits_0^{r(K)} |\partial K_t| dt
\end{equation}
where we used that $|K_{r(K)}| = 0$. Let us introduce the function $f_K \colon t \mapsto |\partial K_t|$. Observe that $f_K \colon t \mapsto |\partial K_t|$ is a strictly decreasing, absolutely continuous function \cite{AbsCont}, and hence the integral \eqref{Eq:IntMa} makes sense (see \cite[p.\,439]{Sch} for the question about differentiability of such functions for general convex bodies).  

For the lens $L$, let $f_L \colon t \mapsto |\partial L_t|$ be an analogous function that measures the area of the inner parallel bodies of $L$. By assumption of the theorem, $f_K(0) = f_L(0)$. The following key proposition relates the values of $f_K$ and $f_L$ in a small right neighborhood of $0$.

\begin{proposition}[Areas of inner parallel bodies are smaller for lenses]
\label{prop_der}
$f_K(t) > f_L(t)$ for every sufficiently small $t > 0$.
\end{proposition}

Before proving Proposition~\ref{prop_der}, let us fix some notation. We label the facets of $K$ by $F_i$. Let $\beta_i$ denote the area of $F_i$. For any adjacent facets $F_i$ and $F_j$, denote their common edge by $E_{ij} = F_i \cap F_j$, its length by $l_{ij}$, and by $\gamma_{ij}$ the angle along $E_{ij}$ between the outward normals to the unit spheres containing $F_i$ and $F_j$ (see Figure~\ref{Fig:Terms}).

\begin{figure}[h]
\begin{center}
\definecolor{ffxfqq}{rgb}{1.,0.4980392156862745,0.}
\definecolor{wqwqwq}{rgb}{0.3764705882352941,0.3764705882352941,0.3764705882352941}
\definecolor{ttzzqq}{rgb}{0.2,0.6,0.}
\begin{tikzpicture}[line cap=round,line join=round,>={Stealth[length=3mm, width=2mm]},x=1.2cm,y=1.2cm]
\clip(-6.,-3.3) rectangle (2.,3.8);
\fill[line width=1.pt,color=ttzzqq,fill=ttzzqq,fill opacity=0.20000000298023224] (-5.930527135815527,0.371267475207329) -- (-3.9997485018428773,-0.14973628316624066) -- (-1.8084679886834414,0.9995367132460453) -- (-1.0576096310274108,3.129522666596815) -- (-4.260250381029664,2.654489828079737) -- cycle;
\fill[line width=2.pt,color=ttzzqq,fill=ttzzqq,fill opacity=0.20000000298023224] (-1.8084679886834414,0.9995367132460453) -- (-0.1228675939454135,-0.5175036420181721) -- (1.6699982804577616,0.21803107568569086) -- (1.9458237995967116,1.0148603531982092) -- (1.593380080696942,2.240751549371314) -- (-1.0576096310274108,3.129522666596815) -- cycle;
\draw [shift={(-1.3797714896663118,2.2156349451517747)},line width=0.2pt,color=ffxfqq,fill=ffxfqq,fill opacity=0.30000001192092896] (0,0) -- (-119.51125452478152:0.7661819976081945) arc (-119.51125452478152:-88.75565205239613:0.7661819976081945) -- cycle;
\draw [line width=2.pt,color=ttzzqq] (-1.0576096310274108,3.129522666596815)-- (-1.8084679886834414,0.9995367132460453);
\draw [line width=0.8pt,dash pattern=on 1pt off 2pt] (-4.045719421699369,-2.494253195847304)-- (-1.0523681571246888,0.3190468648431712);
\draw [line width=0.8pt,dash pattern=on 1pt off 2pt,opacity=0.3] (-1.0523681571246888,0.3190468648431712)-- (0.0303688055762254,1.3366567921936492);
\draw [line width=0.8pt,dash pattern=on 1pt off 2pt] (-1.2721405903577052,-2.739431435081925)-- (-2.8316433446685347,0.46290628178533505);
\draw [line width=0.8pt,dash pattern=on 1pt off 2pt,opacity=0.3] (-2.8316433446685347,0.46290628178533505)-- (-3.2795374240911745,1.3826277120501407);
\draw [line width=1.3pt,dash pattern=on 1pt off 2pt,color=ffxfqq,opacity=0.5] (-2.1777512614008194,0.8058566751075621)-- (-1.3797714896663118,2.2156349451517743);
\draw [line width=1.5pt,dash pattern=on 1pt off 2pt,color=ffxfqq,opacity=0.5] (-1.3797714896663118,2.2156349451517743)-- (-1.3442816605494792,0.5817690179254812);
\draw [line width=1.5pt,color=ffxfqq] (-1.3442816605494792,0.5817690179254812)-- (-1.2721405903577052,-2.739431435081925);
\draw [line width=1.5pt,color=ffxfqq] (-2.1777512614008194,0.8058566751075621)-- (-4.045719421699369,-2.494253195847304);
\draw (-4.398163140599139,1.3060095122893216) node[anchor=north west] {$F_i$};
\draw (0.5513725639497976,0.8156530338200796) node[anchor=north west] {$F_j$};
\draw (-4.244926741077499,-2.5249004757516316) node[anchor=north west] {$o_j$};
\draw (-1.425376989879344,-2.739431435081925) node[anchor=north west] {$o_i$};
\draw [->,line width=0.4pt,color=ttzzqq] (0.12231064528920874,3.313406346022781) -- (-1.17,2.590015129752329);
\draw (-1.7471734288747858,1.4592459118109598) node[anchor=north west] {$\gamma_{ij}$};
\draw (0.22957612495435595,3.727144624731204) node[anchor=north west] {$E_{ij}$};
\draw (-1.701202509018294,2.55) node[anchor=north west] {$a$};
\begin{scriptsize}
\draw [fill=wqwqwq] (-3.2795374240911745,1.3826277120501407) circle (1.0pt);
\draw [fill=wqwqwq] (0.0303688055762254,1.3366567921936492) circle (1.0pt);
\draw [fill=black] (-1.3797714896663118,2.2156349451517743) circle (1.5pt);
\draw [fill=black] (-4.045719421699369,-2.494253195847304) circle (1.5pt);
\draw [fill=black] (-1.2721405903577052,-2.739431435081925) circle (1.5pt);
\end{scriptsize}
\end{tikzpicture}
\caption{
	The facets $F_i$ and $F_j$ are adjacent with $E_{ij} = F_i \cap F_j$ being the common edge; $|F_i| = \beta_i$, $|F_j| = \beta_j$, $|E_{ij}| = l_{\ij}$. The unit spheres containing $F_i$, respectively $F_j$, are centered at $o_i$, respectively $o_j$, and $\gamma_{ij} = \angle o_i a o_j$ is the angle between $F_i$ and $F_j$, where $a \in E_{ij}$ is an arbitrary point.}
\label{Fig:Terms}
\end{center}
\end{figure}
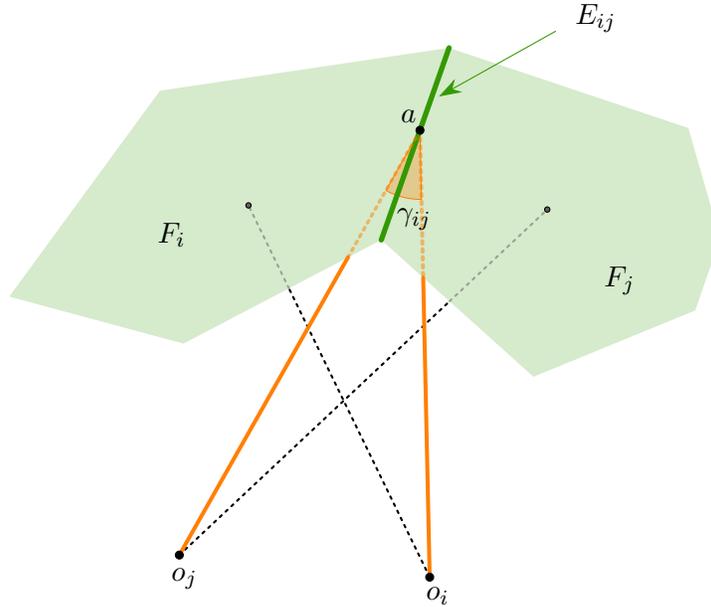

For the proof of Proposition~\ref{prop_der}, we also need the following lemma. Before proceeding, we make the following elementary remark. 

\begin{remark}
If a convex body $M \subset \R^3$ is an intersection of balls of equal radii so that none of them can be removed without changing the body, then any three generating spheres passing through $p \in \partial M$ have linearly independent unit normals at $p$. 

Indeed, if three generating spheres through $p$ have linearly dependent normals, then these normals lie in a plane. Since $M$ has nonempty interior, these normals point strictly away from a common interior direction. In particular, this implies that one lies in the positive cone generated by the other two. Consider the plane through $p$ and the three centers of these generating spheres. In this plane, the generating spheres give three congruent circles through $p$, with one normal lying between the other two. Then the ball corresponding to this middle normal contains the lens formed by the other two balls. In particular, this ball does not provide any additional constraints and can be removed without changing $M$, which contradicts the assumption.  
\end{remark}

\begin{lemma}[Variation of the surface area of $1$-convex polytopes]
\label{surf_area_derivative}
\[
\left.\frac{d}{dt}\right|_{t=0^+} |\partial K_t|  = -2 |\partial K| - 2 \sum_{i<j} l_{ij} \tan \frac{\gamma_{ij}}{2}.
\]
\end{lemma}

\begin{remark}
As will be clear from the proof, a similar formula holds for $1$-convex polytopes in $\mathbb R^n$ for \emph{any} $n \ge 2$, namely,
\[
\left.\frac{d}{dt}\right|_{t=0^+} |\partial K_t|  = -(n-1) |\partial K| - 2 \sum_{i<j} l_{ij} \tan \frac{\gamma_{ij}}{2},
\]
where $\gamma_{ij}$ is the angle between the $i$th and $j$th facets and $l_{ij}$ is the $(n-2)$-dimensional volume of the edge $E_{ij}$. If $n=2$, then $l_{ij} := 1$ and the sum runs over all vertices of the $1$-convex polygon $\partial K$. We will use this formula for $n=2$ in Appendix~\ref{App:RIPDim2}.

For more general deformations in pseudo-Riemannian settings, a similar formula was proven in \cite[Theorem 2]{Cs}. 

We also note that, more generally, if $K$ is a convex body with a smooth boundary, then $\frac{d}{dt}|_{t=0+}|\partial K_t| = \frac{d}{dt}|_{t=0-}|\partial(K+t\B^n)|$ and this quantity is equal, up to a constant, to the total mean curvature of $\partial K$. For $\lambda$-convex polytopes the last equality does not hold.
\end{remark}

\begin{proof}[Proof of Lemma~\ref{surf_area_derivative}]
Since we are interested in the right-sided derivative at $0$, we can assume that for some small $\eps > 0$, there is a one-to-one correspondence between the facets of $\partial K$ and $\partial K_t$. Let $F_i$ be a smooth facet of $\partial K$ and $F_i(t)$ be the corresponding facet of $\partial K_t$. The claim of the lemma will follow once we show  
\begin{equation}
\label{Eq:G}
\left.\frac{d}{dt}\right|_{t=0^+} |F_i(t)| = -2 |F_i| - \sum_j l_{ij} \tan \frac{\gamma_{ij}}{2}.
\end{equation}
Each $F_i(t)$ is a closed convex domain on the sphere of radius $1-t$ centered at $o_i$. Also, $F_i(t)$ is bounded by a finite number of smooth arcs of constant geodesic curvature. We define a convex domain $\tilde F_i(t)$ on the unit sphere centered at $o_i$ in the following way: 
\[
\tilde F_i(t) := \frac{F_i(t) - o_i}{1-t} + o_i.
\]
Thus, $\tilde F_i(0) = F_i$ and $t \mapsto \tilde F_i(t)$ produces a variation of $F_i$ on the unit sphere through a family of convex domains with piecewise smooth boundaries in the sense of Definition~\ref{Def:Def}. The generating spheres of $K_t$ have equal radii and intersect pairwise transversally. By the remark before the statement of Lemma~\ref{surf_area_derivative}, any three of their normals at a common point are linearly independent. Thus the boundary circles on each generating sphere also intersect transversally. This gives the $2$-transversality condition for $\tilde F_i(t)$. 

Note that
\begin{equation}
\label{Eq:Ar}
|\tilde F_i(t)| = \frac{1}{(1-t)^2} |F_i(t)|.
\end{equation}

For every edge $E_{ij}$ of $F_i$, we compute the variational vector field $X_{ij}$ along $E_{ij}$ in $\mathbb R^3$ compatible with the deformation by inner parallel bodies. By symmetry, $X_{ij}(p)$ is orthogonal to $E_{ij}$ at every $p$ in the interior of $E_{ij}$. Let $N_i(q)$ be the outer unit normal to the unit sphere containing $F_i$ at $q \in F_i$; in a similar way, define $N_j(q)$ with respect to $F_j$ for $q \in F_j$. We have that $\left<N_i(p), N_j(p)\right> = \cos \gamma_{ij}$ for each $p \in E_{ij}$. Since the deformation $K_t$ is by inner parallel bodies, the variational vector field along $F_i$, respectively $F_j$, is equal to $-N_i$, respectively $-N_j$. Hence, by straightforward calculation, the variational vector field along the interior of $E_{ij} = F_i \cap F_j$ is 
\[
X_{ij} = -\frac{N_i + N_j}{2 \cos^2 \frac{\gamma_{ij}}{2}}.
\]
The projection of this vector field on the unit sphere containing $F_i$ is the variational vector field along $E_{ij}$ in the deformation $t \mapsto \tilde F_i(t)$. Each vector in this projection is aligned with the corresponding normal to $E_{ij}$ in the sphere and has length $\tan (\gamma_{ij}/2)$. Therefore, by Theorem~\ref{Thm:Variation} applied to the unit $2$-dimensional sphere,
\begin{equation}
\label{Eq:ApplVar}
\left.\frac{d}{dt}\right|_{t=0^+} |\tilde F_i(t)| = -\sum_j \int_{E_{ij}} \tan \frac{\gamma_{ij}}{2} \, \text{d} \sigma_{E_{ij}} = -\sum_{j} l_{ij} \tan \frac{\gamma_{ij}}{2}.
\end{equation}

By \eqref{Eq:Ar},
\[
\left.\frac{d}{dt}\right|_{t=0^+} |\tilde F_i(t)| = 2 |F_i| + \left.\frac{d}{dt}\right|_{t=0^+} |F_i(t)|,
\]
which together with \eqref{Eq:ApplVar} yields \eqref{Eq:G}. 
\end{proof}

\begin{proof}[Proof of Proposition~\ref{prop_der}]
	
	 By Lemma~\ref{surf_area_derivative}, the right-sided derivative of $f_K(t)$ at $t=0$ has the form
	\begin{equation} \label{deriv}
		\left.\frac{d}{dt}\right|_{t=0^+} f_K(t) = -2 f_K(0) - 2 \sum\limits_{i<j} l_{ij} \tan{\frac{\gamma_{ij}}2}.
	\end{equation}
	Observe that \eqref{deriv} similarly holds for the lens $L$.
	
 We want to show that the quantity in \eqref{deriv} is minimized for the lens $L$. Since $ f_K(0) = f_L(0)$, it is enough to show that the term $\sum_{i<j} \tan \frac{\gamma_{ij}}{2} l_{ij}$ is maximal for the lens $L$
	under the following constraint from the Gauss--Bonnet theorem in $\R^3$:
	\begin{equation}\label{Gauss-Bonnet}
		\int\limits_{\partial K} \kappa(z) \, \text{d}\sigma_K(z) = 4\pi.
	\end{equation} 
Here $\kappa(z)$ denotes the Gaussian curvature of $\partial K$ at the point $z$ and the integration is done with respect to the area measure on the boundary of $K$.

Since $\partial K$ is not smooth, this constraint in the case of $\lambda$-convex polytopes should be understood as follows. Recall that, for a point $z \in \partial K$, the \emph{normal cone} of $K$ at $z$ is the subset of $\R^3$ consisting of $0$ and all normal vectors to all possible supporting planes at $z$ for $K$. We denote this cone by $N(K,z)$; this is a closed convex set (see \cite[p.\,81]{Sch}). For a subset $\omega \subseteq \partial K$, we denote by $u(K, \omega)$ the \emph{spherical image} of $\omega$, that is,
$$
u(K, \omega) = \underset{z \in \omega}{\bigcup} N(K, z) \cap \S^{2},
$$
where $\S^2 \subset \R^3$ is the unit sphere centered at the origin (see \cite[p.\,88]{Sch}). Let us denote by $\sigma$ the area measure on $\S^2$. Finally, if $\mathcal{V}$, $\mathcal{E}$ and $\mathfrak{F}$ denote the sets of vertices, relative interiors of edges, and facets of the polytope $K$, then the Gauss--Bonnet theorem for $K$ reads as follows:
	\begin{align}\label{GB}
		\int\limits_{\partial K} \kappa(z) \text{d}\sigma_K(z) = \sum\limits_{F \in \mathfrak{F}}\int\limits_{F} \kappa(z) \text{d}\sigma_K(z) + \sum\limits_{e \in \mathcal{E}} \sigma\left(u(K,e)\right)  + \sum\limits_{v \in \mathcal{V}} \sigma\left(u(K,v)\right) = 4\pi.		
	\end{align} 
\medskip

	To finish the proof, we note that the Gaussian curvature $\kappa$ on each $F \in \mathfrak{F}$ is constant and is equal to $1$. Introduce the standard coordinates $(\theta, \phi)$ on $\S^2$ such that $\text{d} \sigma = \cos \theta \, \text{d} \theta \text{d} \phi$. These coordinates can be further rotated so that the spherical image of the interior $E^\circ_{ij} \in \mathcal E$ of an edge $E_{ij}$ is a piece of the strip in $\S^2$ given by $-{\gamma_{ij}}/2 \le \theta \le {\gamma_{ij}}/2, \ 0 < \varphi < {l_{ij}}/\cos({\gamma_{ij}}/{2})$ in coordinates $(\theta, \phi)$. Therefore,
	\begin{align*}
		 \sigma\left(u(K,E^\circ_{ij})\right) =  \int\limits_{0}^{\frac{l_{ij}}{ \cos{\frac{\gamma_{ij}}2}}}\int\limits_{-\frac{\gamma_{ij}}2}^{\frac{\gamma_{ij}}2} \cos \theta \, \text{d} \theta \text{d} \phi = \frac{l_{ij}}{\cos{\frac{\gamma_{ij}}2}} \cdot 2 \sin{\frac{\gamma_{ij}}2} = 2 l_{ij} \tan{\frac{\gamma_{ij}}2}.
	\end{align*}
	Thus, using \eqref{GB}, 
	\begin{align*}
		\sum_{i} \beta_i + 2 \sum\limits_{i<j} l_{ij} \tan{\frac{\gamma_{ij}}2}  + \sum\limits_{v \in \mathcal{V}} \sigma\left(u(K,v)\right) = 4 \pi;\\
		2 \sum\limits_{i<j}  l_{ij} \tan{\frac{\gamma_{ij}}2} = 4 \pi - \sum_{i} \beta_i	  - \sum\limits_{v \in \mathcal{V}} \sigma\left(u(K,v)\right).
	\end{align*} 
	As before, $\sum_{i} \beta_i$ is the same for $K$ and for the lens $L$. On the other hand, $\sum_{v \in \mathcal{V}} \sigma\left(u(K,v)\right)$ is always nonnegative and equals zero if and only if $K$ has no vertices. Indeed, every vertex belongs to at least three facets. By the remark before Lemma~\ref{surf_area_derivative}, their outward normals are linearly independent. The normal cone contains their positive hull, and thus the spherical image of the normal cone has positive area.
    
    In our setting, the only finite intersection of balls without vertices is the lens $L$. Indeed, if $K$ has no vertices, then every edge has no endpoints, so it is a circle formed by the intersection of the two corresponding spheres. Since $K$ is formed by the intersection of at least two balls of equal radii, each facet of $K$ is homeomorphic to a closed ball. Thus, each facet has a connected boundary (in the corresponding sphere), so it should then consist of only one such edge. Hence, every facet is adjacent to exactly one other facet. Since $K$ is a convex body, its boundary is connected, and so $K$ has exactly two facets which correspond to a lens. 
    
    Therefore, $\sum_{i<j} l_{ij} \tan{\frac{\gamma_{ij}}2} $ is maximal for the lens $L$. This finishes the proof of Proposition~\ref{prop_der}.
\end{proof}

Now we finish the proof of Theorem~\ref{mainthm} in Step I.

By Proposition \ref{prop_der}, $f_K(t) > f_L(t)$ in an open right neighborhood of $t = 0$. We claim that the latter implies that $f_K(t) \ge f_L(t)$ for any $t\in (0, r(K))$, and there are open intervals where this inequality is strict.

If $f_K(t) \ge f_L(t)$ for every $t \in (0, r(K))$, then we are done. Otherwise, assume there is a point $t' \in (0, r(K))$ such that $f_K(t') < f_L(t')$. We choose the largest point $t'' \in (0, t')$ such that $|\partial K_{t''}| = |\partial L_{t''}|$. With this choice, $f_K(t) < f_L(t)$ for each $t \in (t'',t')$. If $K_{t''}$ is not a lens, we can repeat the argument of Proposition \ref{prop_der} (appropriately rescaled to apply to the $(1/(1-t''))$-convex polytope $K_{t''}$) and show that $f_K(t) > f_L(t)$ in an open right neighborhood of $t''$, which is a contradiction.

It is immediate that $K_{t''}$ cannot be a ball. Now we show that $K_{t''}$ cannot be a lens. Indeed, a lens is uniquely determined by its surface area since the surface area is a strictly decreasing function of the distance between the centers of the two balls whose intersection forms the lens. Assume the contrary, i.e., $K_{t''}$ is a lens. Then the lenses $K_{t''}$ and $L_{t''}$ would be equal. Therefore, all but two facets $F', F''$ of the starting polytope $K$ must have disappeared at some moments on the interval $(0,t'')$. Since $K_{t''} = L_{t''}$, the intersection of the unit balls corresponding to $F'$ and $F''$ must be equal to $L$. By assumption, $K$ is not a lens. So we conclude that $K \subset L$, and hence $|\partial K| < |\partial L|$, which is a contradiction.    

We have shown that $|\partial K_t| \ge |\partial L_t|$ for any $t \ge 0$, and this inequality is strict on some open intervals of values of $t$ provided $K$ is not a lens. Therefore, 
$$
|K| = \int\limits_0^{r(K)} |\partial K_t| dt > \int\limits_0^{r(L)} |\partial L_t| dt = |L|.
$$
This finishes the proof in Step I.

\begin{remark}
Note that the above implies that $r(K) \ge r(L)$. Thus, approximating $\lambda$-convex bodies with $\lambda$-convex polytopes, we obtain the following result: among all $\lambda$-convex bodies of given surface area in $\mathbb{R}^3$, the lens has the smallest inradius. This is a partial case of Theorem~\ref{Thm:MainB}, which we prove in Section~\ref{Sec:ProofB} using different methods.   
\end{remark}

\subsection{Step II: General case}
\label{SSec:StepII}

Assume that $K$ is a $\lambda$-convex body which is not a $\lambda$-convex polytope. By Proposition~\ref{Prop:Approx}, we can find a nested sequence $(K^j)_{j=1}^\infty$ of $\lambda$-convex polytopes such that
\[
K \subseteq K^j  \quad \text{ and } \quad \lim_{j \to \infty} d_H(K, K^j) = 0.   
\]
Since $K$ is not a polytope, after passing to a subsequence if needed, we can assume that none of the approximating polytopes $K^j$ is a lens (otherwise, a subsequence of lenses would converge to $K$ so that $K$ would be a lens). By the continuity of volume and surface area functionals under Hausdorff limits for convex bodies,
\[
\lim_{j \to \infty} |K^j| = |K| \quad \text{ and }\quad \lim_{j \to \infty} |\partial K^j| = |\partial K|.
\]
By Step I, if $L^j$ is the corresponding sequence of $\lambda$-convex lenses centered at the origin $o$ and symmetric with respect to a fixed plane $\Omega$ passing through $o$, and such that $|\partial K^j| = |\partial L^j|$, then
\begin{equation}
\label{Eq:Eq}
|K^j| > |L^j|.
\end{equation}
Moreover, $L^j \supseteq L^{j+1}$ by construction. The Hausdorff limit of $(L^j)_{j=1}^\infty$ is the lens $L$ with surface area equal to $|\partial K|$. Indeed, for each $j$, the centers of two balls of radius $1/\lambda$ forming the facets of $L^j$ lie on the line passing through $o$ orthogonal to $\Omega$. These centers are symmetric with respect to $o$ and their distance to $\Omega$ is uniquely defined by the surface area $|\partial L^j|$. Since the sequence $(|\partial L^j|)_{j=1}^\infty$ converges to $|\partial K|$ as $j \to \infty$, the sequences of the corresponding centers converge to some limiting pair of points symmetric with respect to $\Omega$. This pair of points uniquely defines the lens $L$; it is the Hausdorff limit of $(L^j)_{j=1}^\infty$ (see, e.g., \cite[Theorem 1.8.10]{Sch}) and $|\partial L| = |\partial K|$.    

Thus, the required inequality between the three-dimensional volumes of $K$ and $L$ follows from \eqref{Eq:Eq} by passing to the limit. To study the equality case, it is enough to show that we must have strict inequality in the limit. 

For each $j \ge 1$, consider the graph of the function $f_{K^j}(t) = |\partial K^j_t|$, $t < r(K)$. By Lemma~\ref{Lem:Des}, 
\[
\lim_{j \to \infty} d_H(K^j_t, K_t) = 0.
\]
By continuity of surface area, 
\[
\lim_{j \to \infty} |\partial K_t^j| = |\partial K_t|.
\]
Therefore, the graph of the function $t \mapsto f_{K^j}(t)$ converges pointwise to the graph of the function $t \mapsto f_K(t) = |\partial K_t|$. The same is true for the functions $f_{L^j}$ and $f_{L}$ associated to the lenses. From Step I, $f_{K^j}(t) \ge f_{L^j}(t).$ 
Hence, 
\begin{equation}
\label{Eq:EqEq}
f_{K}(t) \ge f_{L}(t)
\end{equation}
for all $t$ in the common domain of definition.

The graph of $f_K$ intersects the $t$-axis only at $r(K)$. Since $K$ is not a $\lambda$-convex polytope, Proposition~\ref{Prop:Reduction} ensures that there exists a $\lambda$-convex polytope $K''$ so that $|\partial K| = |\partial K''|$ and $r(K) > r(K'')$. Since $|\partial K| = |\partial L|$, we have $|\partial L| = |\partial K''|$, and thus $r(K'') \ge r(L)$ by the remark at the end of Step I. Thus, $r(K) > r(L)$, and therefore the graphs of $f_K$ and $f_L$ cannot coincide. Because of \eqref{Eq:EqEq}, this means that the area under $f_K$, i.e., $|K|$, must be strictly larger than the area under $f_L$, i.e.,  $|L|$. We showed that $|K| > |L|$, and this concludes the proof of the equality case and the proof of Theorem~\ref{mainthm}.

\section{Proof of the reverse inradius inequality (Theorem~\ref{Thm:MainB})}
\label{Sec:ProofB}

In this section we work in model spaces $M^n(c)$. In each model space, the proof of Theorem~\ref{Thm:MainB} follows the same lines, except for the equidistant case in the hyperbolic space, where an additional argument is required. We describe this argument in the next subsection, and after that reduce Theorem~\ref{Thm:MainB} to an equivalent statement.

\subsection{Reduction to an equivalent statement (Theorem~\ref{Thm:MainC})}

Assume $c < 0$ and $0 < \lambda < \sqrt{-c}$, i.e., we are in the equidistant case (see \eqref{It:Horo}, Section~\ref{SSec:Facts}). Since any $\lambda$-convex lens $L$ is compact by definition, we must have $r(L) < \rr_\lambda$, where $\rr_\lambda$ is the characteristic distance in the equidistant case. If $r(K) \ge \rr_\lambda$, then $r(K) \ge r(L)$ and  the claim of the theorem follows. 
	
Therefore, in the rest of this section we consider $\lambda$-convex bodies $K \subset M^n(c), \ c\in \R$ such that
\begin{equation}
\label{Eq:AssEquidistant}
\text{either $c \ge 0$; \ \  or $c < 0$ and $\lambda \ge \sqrt{-c}$;\ \  or $c<0$ and $\lambda \in (0, \sqrt{-c})$ with $r(K) < \rr_\lambda.$}
\end{equation}

Observe that for every $n \ge 2$, every $c \in \R$ and $\lambda$-convex body $K \subset M^n(c)$ as in~\eqref{Eq:AssEquidistant}, there exists a $\lambda$-convex lens $L \subset M^n(c)$ such that $r(K) = r(L)$ (and it is unique up to rigid motions in the space).

\begin{MainTheorem}
\label{Thm:MainC}
Let $c \in \R$ and $K \subset M^n(c)$ be a $\lambda$-convex body as in \eqref{Eq:AssEquidistant}. If $L \subset M^n(c)$ is a $\lambda$-convex lens such that $r(L) = r(K)$, then $|\partial L| \ge |\partial K|$. Moreover, equality holds if and only if $K$ is a $\lambda$-convex lens. 
\end{MainTheorem}

In fact, we have the following:

\begin{claim}
\label{Claim:Equivalent}
For a $\lambda$-convex body $K \subset M^n(c)$ and $c \in \R$ as in \eqref{Eq:AssEquidistant}, Theorems~\ref{Thm:MainB} and \ref{Thm:MainC} are equivalent.
\end{claim}

\begin{proof}[Proof of Claim~\ref{Claim:Equivalent}]
We use a monotonicity argument. Suppose Theorem~\ref{Thm:MainB} is true. Let $K$ and $L$ be as in the assumption of Theorem~\ref{Thm:MainC}, i.e., $r(L)=r(K)$. Let $L'$ be the $\lambda$-convex lens such that $|\partial L'| = |\partial K|$. By Theorem~\ref{Thm:MainB}, $r(L') \le r(K)$. Hence, $r(L') \le r(L)$. The surface area of a lens grows monotonically as its inradius increases. Thus, $|\partial L'| \le |\partial L|$. We conclude that $|\partial K| \le |\partial L|$, which yields Theorem~\ref{Thm:MainC}.

In the other direction, suppose Theorem~\ref{Thm:MainC} is true. Let $K$ and $L$ be as in the assumption of Theorem~\ref{Thm:MainB}, i.e., $|\partial L| = |\partial K|$. Let $L''$ be a $\lambda$-convex lens such that $r(L'') = r(K)$. Since~\eqref{Eq:AssEquidistant} holds, such a lens $L''$ always exists. By Theorem~\ref{Thm:MainC}, $|\partial L''| \ge |\partial K|$. Hence, $|\partial L''| \ge |\partial L|$, and $r(L'') \ge r(L)$ by monotonicity. Therefore, $r(K) \ge r(L)$, as asserted by Theorem~\ref{Thm:MainB}.
\end{proof}

In view of Claim~\ref{Claim:Equivalent} and our discussion at the beginning of the section, to prove Theorem~\ref{Thm:MainB} it is enough to establish Theorem~\ref{Thm:MainC}. The rest of the section will be dedicated to the proof of Theorem~\ref{Thm:MainC}. The proof will be based on two claims, Claims~\ref{Claim2} and \ref{KeyClaim} (Key Claim).

\subsection{Proof of Theorem~\ref{Thm:MainC}}

By the first part of Proposition~\ref{Prop:Reduction}, we can assume that $K$ is a $\lambda$-convex polytope.

If $K$ is a ball, the claim of the theorem follows immediately. So in the remainder of this section we implicitly assume that $K$, and hence $L$, are not balls.

Fix the inradius $r = r(K) = r(L) > 0$.  Let $B \subset K$ be an inscribed ball for $K$ (of radius $r$), and denote by $o$ the center of $B$. Consider the radial projection  
\begin{equation}\label{Eq::radial}
	\Pi \colon M^n(c) \sm \{o\} \to \partial B
\end{equation}
from the center of $B$ onto its boundary. 

We denote the facets of $K$ as $F_1, \ldots, F_\ell$, and let $\Ss_1, \ldots, \Ss_\ell$ be the totally umbilical hypersurfaces containing the corresponding facets, $F_i \subset \Ss_i$. As usual, we denote by $\Bb_i$ the convex region bounded by $\Ss_i$. We have
\[
K = \bigcap_{i=1}^\ell \Bb_i.
\]

\begin{claim}
\label{Claim2}
Assume that there exist facets of $K$ that do not touch $B$, and let $\mathcal N \subset \{ 1, \dots, \ell\}$ be the set of their indices. If 
\[
K' = \bigcap_{i=1, \, i  \not\in \mathcal N}^\ell \Bb_i,
\]
then
\[
r(K') = r(K) \quad \text{ and } \quad |\partial K| < |\partial K'|.
\]
\end{claim}
  
\begin{proof}
Since $B$ is the inscribed ball for $K$, the set $\partial B \cap \partial K$ is not contained in any open half-space with respect to a hyperplane passing through $o$ (Proposition~\ref{Prop:Char}). By construction, $B \subset K'$ and $\partial B \cap \partial K \subseteq \partial B \cap \partial K'$. Hence, $\partial B \cap \partial K'$ does not lie in any open half-space through $o$. Moreover, each facet of $K'$ touches $B$. 

Now observe that $K'$ is compact. This follows by Lemma~\ref{Lem:Cpt2} in the equidistant case, and by Lemma~\ref{Lem:Cpt1} in all other cases. Hence, $K'$ is a $\lambda$-convex polytope. 

By construction, $K \subset K'$, so we obtain $|\partial K| < |\partial K'|$. 
Applying Proposition~\ref{Prop:Char}, we conclude that $B$ is an inscribed ball of $K'$. Thus, $r(K') = r(K)$.  
\end{proof}

To formulate the next claim, we introduce some additional notation. By Claim~\ref{Claim2}, we can assume that each $F_i$ touches $B$ at some point $p_i$. For any two adjacent facets $F_i$ and $F_j$, $\Pi$ maps $F_i \cap F_j$  into a great $(n-2)$-dimensional sphere of $\partial B$ since the intersection of $\Ss_i$ and $\Ss_j$ lies in a totally geodesic hyperplane through $o$. Fix $i \in \{1,\ldots,\ell\}$. Let $\delta_i$ be a totally geodesic hyperplane passing through $o$ perpendicularly to the geodesic $op_i$, and let $\mathcal H_i$ be the closed half-space bounded by $\delta_i$ and containing the point $p_i$. Consider a \emph{natural radial extension} $F$ of the facet $F_i$ defined as $F := \Ss_i \cap \mathcal H_i$. Note that $F_i \subseteq F.$

\begin{wrapfigure}{r}{0.38\textwidth}
\definecolor{ffqqff}{rgb}{1.,0.,1.}
\definecolor{qqzzff}{rgb}{0.,0.6,1.}
\definecolor{qqqqff}{rgb}{0.,0.,1.}
\definecolor{ffxfqq}{rgb}{1.,0.4980392156862745,0.}
\definecolor{ttzzqq}{rgb}{0.2,0.6,0.}
\begin{tikzpicture}[line cap=round,line join=round,>=triangle 45,x=1.0cm,y=1.0cm,scale=0.8]
\clip(-3.2,-3.1) rectangle (3.5,3.1);
\draw [shift={(0.,0.)},line width=0.1pt,color=ttzzqq,fill=ttzzqq,fill opacity=0.10000000149011612]  (0,0) --  plot[domain=0.:3.141592653589793,variable=\t]({1.*3.*cos(\t r)+0.*3.*sin(\t r)},{0.*3.*cos(\t r)+1.*3.*sin(\t r)}) -- cycle ;
\draw [line width=1.1pt,color=ttzzqq] (-3.,0.)-- (3.,0.);
\draw [line width=0.8pt] (0.,0.) circle (3.cm);
\draw [shift={(0.,-6.416442619084425)},line width=1.1pt,color=qqqqff]  plot[domain=1.0842805504705857:2.057312103119208,variable=\t]({1.*5.671925236108369*cos(\t r)+0.*5.671925236108369*sin(\t r)},{0.*5.671925236108369*cos(\t r)+1.*5.671925236108369*sin(\t r)});
\draw [shift={(0.,-1.2577608269436564)},line width=1.1pt,color=qqzzff]  plot[domain=-0.05458032076234964:3.1961729743521428,variable=\t]({1.*2.6558567282357988*cos(\t r)+0.*2.6558567282357988*sin(\t r)},{0.*2.6558567282357988*cos(\t r)+1.*2.6558567282357988*sin(\t r)});
\draw [line width=0.8pt] (0.,1.3980959012921426)-- (0.,0.);
\draw [line width=0.8pt,color=ffqqff] (0.,0.)-- (0.,-0.7445173829760552);
\begin{scriptsize}
\draw [color=ffxfqq](2.6558444474344136,-1.2) node[anchor=north west] {$Q$};
\draw [color=ffxfqq](-3.2,-1.2) node[anchor=north west] {$Q$};
\draw [color=qqqqff](-1.5,-1) node[anchor=north west] {$T$};
\draw [color=qqzzff](1.9,-0.1315101948319532) node[anchor=north west] {$\mathcal S_i$};
\draw [color=ttzzqq](-1.8921965362934485,2.1506609081139154) node[anchor=north west] {$\mathcal H_i$};
\draw (0,0.9) node[anchor=north west] {$r$};
\draw [color=ffqqff](0,-0.1) node[anchor=north west] {$\mathfrak{r}_\lambda-r$};
\draw (1.3,3.2) node[anchor=north west] {$M^n(c),c<0$};
\draw [color=ffxfqq, fill=ffxfqq] (-2.6519017964432647,-1.4026463781085334) circle (2.5pt);
\draw [color=ffxfqq, fill=ffxfqq] (2.6519017964432643,-1.4026463781085339) circle (2.5pt);
\draw [fill=black] (0.,0.) circle (1.5pt);
\draw [fill=black] (0.,1.3980959012921426) circle (1.5pt);
\draw [fill=black] (0.,-0.7445173829760552) circle (1.5pt);
\end{scriptsize}
\end{tikzpicture}
\caption{}
\label{Fig:NRE}
\end{wrapfigure}

\medskip

Observe that the natural radial extension $F$ is compact. Indeed, for $c \ge 0$, or $c < 0$ and $\lambda > \sqrt{-c}$, this is true since in these cases $\mathcal S_i$ is a geodesic sphere, and hence compact. If $c < 0$ and $\lambda = \sqrt{-c}$, then $\Ss_i$ is a horosphere. In this case, choose the coordinates in the Poincar\'e unit ball model so that $\mathcal H_i$ is a closed upper half-ball centered around the north pole and $\mathcal S_i$ is a Euclidean sphere inside the unit ball touching the south pole. It becomes clear that $\Ss_i \cap \mathcal H_i$ is a compact set.

Finally, if $c < 0$ and $\lambda < \sqrt{-c}$, then $\Ss_i$ is an equidistant. Again, choose the coordinates in the ball model so that $\mathcal H_i$ is a closed upper half-ball centered around the north pole and $\Ss_i$ is the Euclidean sphere that intersects the boundary of the unit ball along the small sphere $Q$ parallel to the equatorial sphere (see Figure~\ref{Fig:NRE}). Let $T$ be the hyperbolic hyperplane spanning $Q$. Then $\Ss_i$ is a connected component of the equidistant set for $T$ at distance $\rr_\lambda$. Since $r < \rr_\lambda$, $Q$ must lie in the open southern hemisphere of the unit ball. Therefore, $\Ss_i \cap \mathcal H_i$ is compact.

Since $F$ is compact, the $(n-1)$-dimensional volume of $F$ is well-defined and we observe that 
\begin{equation}
\label{Eq:Half}
|F| = \frac{1}{2} \cdot |\partial L|.
\end{equation}
(Here and everywhere below, we emphasize that $|\cdot|$ denotes the $(n-1)$-dimensional volume.) The above holds since the reflection of $\Ss_i$ in $\delta_i$, together with $\Ss_i$, bounds the lens of inradius $r$ whose boundary consists of two copies of $F$. Define
\[
\tF_i = \Pi(F_i), \quad \tF = \Pi(F).
\]
Note that $\tF$ is the half-sphere. Also,  $\widetilde F_i$ is spherically convex, and hence star-shaped with respect to $p_i$. We have the following trivial observation:
\begin{equation}
\label{Eq:Claim3}
\sum_{i=1}^\ell |\tF_i| = \left|\partial B\right| = 2 \cdot |\tF|.
\end{equation}
Indeed,  $\Pi$ is a bijection from $\partial K$ onto $\partial B$, and the lower-dimensional faces do not contribute to the sum since they have measure zero.

\begin{claim}[Key Claim]
\label{KeyClaim}
For each $i \in \{1, \ldots, \ell\}$,
\begin{equation}
\label{Eq:KeyClaim}
\frac{\left|F_i\right|}{|\tF_i|} \le \frac{\left|F\right|}{|\tF|},
\end{equation}
and equality is possible if and only if $F_i = F$.
\end{claim}

We postpone the proof of the Key Claim to the next subsection. Instead, we show how Claim~\ref{KeyClaim} together with \eqref{Eq:Claim3} imply Theorem~\ref{Thm:MainC}. We have, 
\begin{equation*}
\begin{aligned}
|\partial K| &= \sum_{i=1}^\ell |F_i| \overset{\eqref{Eq:KeyClaim}}{\le} \sum_{i=1}^\ell |\tF_i| \cdot \frac{|F|}{|\tF|} \\
&= \frac{|F|}{|\tF|} \cdot \sum_{i=1}^\ell |\tF_i| \overset{{\eqref{Eq:Claim3}}}{=} \frac{|F|}{|\tF|} \cdot 2 |\tF| = 2 |F| \\
& \overset{\eqref{Eq:Half}}{=} \left|\partial L\right|.
\end{aligned}
\end{equation*}
To finish the proof, it remains to analyze the equality case. If $|\partial K| = |\partial L|$, then there must be equality in the chain of inequalities above. By Claim~\ref{KeyClaim}, this is only possible if all the facets of $K$ are equal to one of the facets of the lens. Then  $|\tF_i| = {|\partial B|}/2$ so $\ell =2$ by \eqref{Eq:Claim3}. This implies that $K$ must be a lens itself, and thus Theorem~\ref{Thm:MainC} follows.   

\subsection{Proof of the Key Claim}

For the proof of Claim~\ref{KeyClaim}, we need the following auxiliary result, the proof of which is straightforward and thus omitted. 

\begin{lemma}
	\label{Claim1}
	Let $q \colon [0, 1] \to \R$ be a continuous and almost everywhere positive function, and $f \colon [0, 1] \to \R$ be a continuous, strictly increasing function. If  
	\[
	\int \limits_0^1 f(t) q(t) dt = 0,
	\]
	then 
	\begin{equation}
		\label{Eq:Claim}
		\int \limits_0^x f(t) q(t) dt \le 0 \quad \text{for all } x \in (0,1],
	\end{equation}
	and this inequality is strict unless $x = 1$. 
	\end{lemma}

In order to compute various quantities in all model spaces in a uniform way, it will be useful to introduce the generalized sine function
\begin{equation*}
\sn_c(t) := \left\{
\begin{aligned}
    &\frac{1}{\sqrt{c}} \sin\left(t \sqrt{c}\right),& \qquad &c>0\\
    &t,& \qquad &c=0\\
    &\frac{1}{\sqrt{-c}} \sinh\left(t \sqrt{-c}\right),& \qquad &c<0
\end{aligned}
\right.,
\end{equation*}
and the generalized cosine function $\cn_c(t) := \sn_c'(t)$. In this notation, the intrinsic diameter of the sphere $\partial B$ is equal to $\pi \sn_c(r)$. We set
\[
t_0 := \frac{\pi}{2}\sn_c(r).
\]

Assume that the facet $F_i$ touches the inscribed ball $B$ at some point $p_i$. We introduce and compare polar coordinates on $F_i$ and on $\partial B$ centered at $p_i$ as follows. 

Let $(t, \theta)$ be polar coordinates on the sphere $\partial B$ centered at $p_i$, where $\theta$ is the angular direction measured in $T_{p_i} \partial B$ and $t$ is the radial distance from the point $p_i$. In a similar way, we introduce polar coordinates $(\tau, \theta)$ on the totally umbilical hypersurface $\partial \mathcal B_i$ that contains $F_i$, also centered at $p_i$, where $\theta$ is the angular direction measured in $T_{p_i} \partial \mathcal B_i$ and $\tau$ is the radial distance from $p_i$. Since $T_{p_i} \partial \mathcal B_i = T_{p_i} \partial B$, we can assume that the angular direction in both coordinate systems is measured along the same sphere $\mathbb S^{n-2}$ (hence, we used the same notation for the coordinate). 

In this way, each point $x = (\tau, \theta) \in F_i \subset \partial \mathcal B_i$ corresponds to a unique point $\tilde x = (t, \theta) \in \partial B$ with $\tilde x = \Pi(x)$. The angular coordinates of $x$ and $\tilde x$ are the same because $\tilde x$ is obtained from  $x$ by the radial projection, and $\tau = \rho(t)$ for some smooth function $\rho$.

\begin{claim}\label{Claim::rho}
  The map $t \mapsto \rho(t)$ is strictly increasing and strictly convex on $[0,t_0]$.
\end{claim} 

\begin{proof}[Proof of \textsc{Claim} \ref{Claim::rho}]
By rotational symmetry, it is enough to prove this claim in dimension $2$, and hence we assume that $\partial B$ is a circle of radius $r$, and $F_i$ is a curve of constant geodesic curvature  $\lambda$ touching $\partial B$ at $p_i$ (see Figure~\ref{Fig:Rho}).

    \begin{figure}
        \centering
        \includegraphics[width=0.9\linewidth]{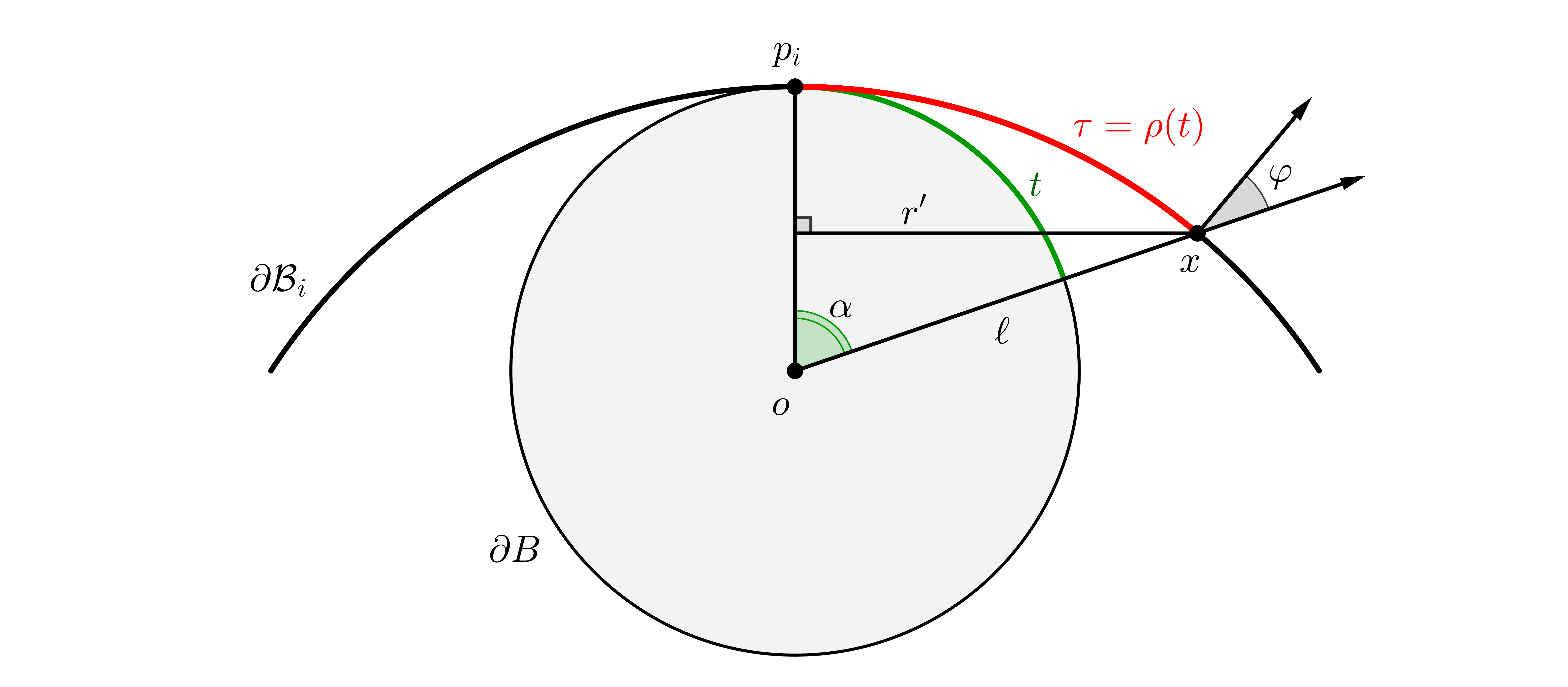}
        \caption{}
        \label{Fig:Rho}
    \end{figure}

Let $x \in F_i$ be a point at distance $\tau = \rho(t)$ from $p_i$ for some $t \in [0, t_0]$, and let $\alpha$ be the smallest angle between the geodesic rays $op_i$ and $ox$. We have that 
\begin{equation}
\label{Eq:Alpha}
\alpha = \frac{t}{\sn_c(r)}.
\end{equation}
We denote by $\phi$ the angle between the ray $ox$ and the outward-facing normal to $F_i$ at the point $x$. Finally, let $\ell = \dist_{M^2(c)}(o,x)$.

By \cite[Proposition 3.1]{DrBla} and the computations in the proof of that proposition (simplified to the setting of $2$-dimensional model spaces, the curve $F_i$ of geodesic curvature $\lambda$, and the central point $o$), we have the following formulas: 
\begin{equation}
  \label{Eq:BorF}
\frac{d\alpha}{d\tau} = \frac{\cos \phi}{\sn_c(\ell)}, \qquad \frac{d\ell}{d\tau} = \sin \phi, \qquad \lambda = - \frac{d\phi}{d\tau} + \frac{\cn_c(\ell)}{\sn_c(\ell)} \cdot \cos\phi.  
\end{equation}
Therefore, using \eqref{Eq:Alpha} and the first formula in \eqref{Eq:BorF},
\begin{equation}
\label{Eq:RhoDer}
\rho'(t) = \frac{d\tau}{dt} = \frac{d \tau}{d\alpha} \cdot \frac{d\alpha}{dt}= \frac{1}{\sn_c(r)} \cdot \frac{\sn_c(\ell)}{\cos \phi}
\end{equation}
(here, we implicitly assume that $\ell = \ell(t)$ and $\phi=\phi(t)$).
Since the last quantity in \eqref{Eq:RhoDer} is positive for $t \in [0,t_0]$, we conclude that $t \mapsto \rho(t)$ is a strictly increasing function in the same range. In order to conclude that $t \mapsto \rho(t)$ is strictly convex, it is enough to show that $\ell$ and $\phi$ are strictly increasing as functions of $t$. But since $\tau = \rho(t)$ is strictly increasing in $t$, it is then enough to show that $\ell$ and $\phi$ are strictly increasing as functions of $\tau$ (on the domain that corresponds to $[0,t_0]$).

The claim for $\ell$ follows immediately from the second formula in \eqref{Eq:BorF}.

The claim for $\phi$ can be verified using the third formula in \eqref{Eq:BorF}. Indeed, denote by $R$ the quantity that verifies
\begin{equation*}
    \begin{aligned}
        \frac{\cn_c(R)}{\sn_c(R)} = \lambda& \qquad &\text{if}& \quad \text{$c < 0$ and $\lambda > \sqrt{-c}$}, \qquad \text{or} \quad \text{$c \ge 0$},\\
        -c \cdot \frac{\sn_c(R)}{\cn_c(R)} = \lambda& \qquad &\text{if}& \quad \text{$c < 0$ and $\lambda \in (0, \sqrt{-c})$}.\\
    \end{aligned}
\end{equation*}
In the first case, $R$ is just the radius of a circle of curvature $\lambda$. With this notation, a direct computation gives
\begin{equation*}
    \frac{\cn_c(\ell)}{\sn_c(\ell)} \cdot \cos\phi - \lambda=\left\{
\begin{aligned}
    &\frac{\sn_c(R-r)}{\sn_c(R)}\cdot \frac{\cos \alpha}{\sn_c(\ell)}&\qquad &\text{if}& \quad &\text{$c < 0$ and $\lambda > \sqrt{-c}$, or $c \ge 0$},\\
    &e^{-r\sqrt{-c}} \cdot \frac{\cos \alpha}{\sn_c(\ell)}& \qquad &\text{if}& \quad &\text{$c <0$ and $\lambda = \sqrt{-c}$,}& \\
    &\frac{\cn_c(R-r)}{\cn_c(R)}\cdot \frac{\cos \alpha}{\sn_c(\ell)}& \qquad &\text{if}& \quad &\text{$c <0$ and $\lambda \in \left(0, \sqrt{-c}\right)$.}&\\
\end{aligned}
    \right.
\end{equation*}
In all cases, by using the third formula in~\eqref{Eq:BorF}, we conclude that $d \phi/d\tau > 0$ (except for the boundary point $\alpha = \pi/2$,  where the claim is obvious). This concludes the proof of Claim~\ref{Claim::rho}.
\end{proof} 

Now we return to the proof of Claim~\ref{KeyClaim}. 

We first treat the case $n=2$. Let $\varphi_-,\varphi_+\in(0,t_0]$ be the lengths of the two subarcs of $\widetilde F_i$ from $p_i$ to its endpoints. Since $\varrho(0)=0$ and $\varrho$ is strictly convex by Claim~\ref{Claim::rho}, we have
\[
\varrho(s)\leqslant \frac{s}{t_0}\varrho(t_0), \qquad 0\leqslant s\leqslant t_0,
\]
with strict inequality for $0<s<t_0$. Therefore,
\[
\frac{|F_i|}{|\widetilde F_i|} = \frac{\varrho(\varphi_-)+\varrho(\varphi_+)}{\varphi_-+\varphi_+} \le \frac{\varrho(t_0)}{t_0} = \frac{|F|}{|\widetilde F|}.
\]
Equality holds if and only if $\varphi_-=\varphi_+=t_0$, equivalently, $F_i=F$. This proves Claim~\ref{KeyClaim} for $n=2$, including its equality case. For the remainder of this subsection, assume $n\ge 3$.

By the Gauss equation, the intrinsic sectional curvature of $\partial\mathcal B_i$ is $\kappa=c+\lambda^2$. Thus $\partial\mathcal B_i$ is intrinsically a model space $M^{n-1}(\kappa)$ of dimension $n-1$. Namely, for $c \ge 0$ and $\lambda > 0$, we have $\kappa > 0$ and thus $\partial \mathcal B_i$ is the spherical space. The same is true when $c < 0$ and $\lambda > \sqrt{-c}$. When $c < 0$ and $\lambda = \sqrt{-c}$, then $\partial \mathcal B_i$ is intrinsically  the Euclidean space. Finally,  $\partial \mathcal B_i$ is intrinsically the hyperbolic space when $c < 0$ and $\lambda \in (0, \sqrt{-c})$.

In each of these cases, the area element on $\partial \mathcal B_i$ can be written in polar coordinates as $\tilde g(\tau) d\tau d\theta$, where $d \theta$ is the standard area measure on $\mathbb S^{n-2}$, and $\tilde g(\tau)$ is a non-decreasing function in $\tau$ explicitly given by
\begin{equation*}
\tilde g(\tau) = \left(\sn_\kappa(\tau)\right)^{n-2}.
\end{equation*}

Therefore, in the coordinates $(t, \theta)$, the area element on $\partial \mathcal B_i$ can be written as $\widehat g(t) dt d\theta$,  where 
\[
\widehat g(t) = \tilde g(\rho(t)) \cdot \rho'(t) = \big(\sn_\kappa(\rho(t))\big)^{n-2} \cdot \rho'(t). 
\]
By Claim~\ref{Claim::rho}, $\rho'(t)$ is strictly increasing and $\rho(t)$ is monotone, and hence the density $\widehat g(t)$ is strictly increasing on $[0,t_0]$ for all values of $c$ and $\lambda >0$. 

In the same way, the intrinsic geometry of the inscribed sphere $\partial B$ is spherical. Hence, the area element on $\partial B$ can be written as $q(t) dt d\theta$ where 
\[
q(t) = \left(\sn_c(r) \sin\left(\frac{t}{\sn_c(r)}\right)\right)^{n-2}.
\]
Observe that $q(t) = \left(\sn_{{1}/{(\sn_c(r))^2}}(t)\right)^{n-2}$ since the intrinsic sectional curvature of $\partial B$ is $\frac{1}{(\sn_c(r))^2}$.

\begin{claim}\label{Claim::g}
  The function $t \mapsto g(t) := \frac{\widehat g(t)}{q(t)}$ is smooth and strictly increasing for $t \in [0, t_0]$.
\end{claim}

\begin{proof}[Proof of \textsc{Claim}~\ref{Claim::g}]
    Observe that 
    \[
    g(t) = \left(\frac{1}{\sn_c(r)} \cdot \frac{\sn_\kappa(\rho(t))}{\sin\left(\frac{t}{\sn_c(r)}\right)}\right)^{n-2} \cdot \rho'(t).
    \]
    Since $\rho'$ is positive and strictly increasing by Claim~\ref{Claim::rho}, it remains to show that the function $t \mapsto \sn_\kappa(\rho(t)) / \sin\left(\frac{t}{\sn_c(r)}\right)$ is strictly increasing for $t \in [0,t_0]$. 

    In the notation of Figure~\ref{Fig:Rho}, consider any totally geodesic $3$-dimensional subspace $M^3(c) \subset M^n(c)$ passing through $p_i, o$, and $x$. In $M^3(c)$, consider the geodesic circle $\mathcal C$ obtained by rotating $x$ around the axis $o p_i$. Now we can compute the normalized length $|\mathcal C| / 2\pi$ in two ways: intrinsically as a circle in $\partial \mathcal B_i \cap M^3(c)$ (which is a totally umbilical hypersurface of curvature $\lambda$), and extrinsically in $M^3(c)$. Since intrinsically in $\partial \mathcal B_i \cap M^3(c)$ the radius of $\mathcal C$ is equal to $\rho(t)$, we obtain
    \[
    \frac{1}{2\pi}|\mathcal C| = \sn_\kappa(\rho(t)).   
    \]
    On the other hand, the radius of $\mathcal C$ as a circle in $M^3(c)$ is $r' := \sn^{-1}_c(\sn_c(\ell) \sin \alpha)$, so we get
    \[
    \frac{1}{2\pi}|\mathcal C| = \sn_c(r') = \sn_c(\ell) \sin \alpha = \sn_c(\ell) \sin\left(\frac{t}{\sn_c(r)}\right).
    \]
    It follows that
    \begin{equation}\label{Eq:identity}
        \sn_\kappa(\rho(t)) = \sn_c(\ell) \sin\left(\frac{t}{\sn_c(r)}\right),
    \end{equation}
    and therefore the function simplifies to $t \mapsto \sn_\kappa(\rho(t)) / \sin\left(\frac{t}{\sn_c(r)}\right) = \sn_c(\ell)$. The required monotonicity now follows from the discussion in Claim~\ref{Claim::rho}.  
\end{proof}

We note that in the case $\kappa > 0$, $\rho(t_0)\leqslant \frac{\pi}{2\sqrt{\kappa}}$ by \eqref{Eq:identity}. Thus, $\tilde g(\tau)$ is strictly increasing for $\tau \in [0, \rho(t_0)]$.

Using the discussed area elements, we can now write
\begin{equation*}
|F_i| = \int \limits_{\tF_i} g(t) q(t) dt d\theta = \int \limits_{\S^{n-2}} \int \limits_{0}^{\phi(\theta)} g(t) q(t) dt d\theta, \qquad |\tF_i| = \int \limits_{\tF_i} q(t) dt d\theta = \int \limits_{\S^{n-2}} \int \limits_{0}^{\phi(\theta)} q(t) dt d\theta,
\end{equation*}
where $\phi(\theta) := \max \{t \in [0, t_0] : (t,\theta) \in \tilde F_i\}$.

Furthermore, the rotational symmetry of the facet $F$ of the lens and the fact that the projection $\tF$ is equal to the hemisphere of $\partial B$ containing $p_i$ imply
\begin{equation}
\label{Eq:Par}
|F| = \int\limits_{\S^{n-2}} \int\limits_{0}^{{t_0}} g(t) q(t) dt d\theta = |\S^{n-2}| \int\limits_{0}^{{t_0}} g(t) q(t) dt, \qquad |\tF| = |\S^{n-2}| \int\limits_{0}^{{t_0}} q(t) dt.
\end{equation}

Consider the function
$$
\mathcal{R}(s) := \int\limits_{0}^s \left(g(t) - \frac{|F|}{|\widetilde{F}|} \right) q(t) dt
$$
for $s \in (0, {t_0}]$. Since $t \mapsto g(t)$ is strictly increasing on $[0, {t_0}]$  by Claim~\ref{Claim::g}, we get that $t \mapsto g(t) - \frac{|F|}{|\widetilde{F}|} $ is also strictly increasing  on $[0, {t_0}]$ as $\frac{|F|}{|\widetilde{F}|}$ is a constant independent of $t$. 
Moreover, by \eqref{Eq:Par}, $\mathcal{R}(t_0) = 0$. Hence,  by Lemma~\ref{Claim1}, $\mathcal{R}(s) \le 0$  for any $s \in (0, t_0]$ with equality if and only if $s = t_0$. Therefore,
\begin{align*}
    |F_i| =  \int\limits_{\S^{n-2}} \int\limits_{0}^{\varphi (\theta)} g(t) q(t) dt d\theta \le  \int\limits_{\S^{n-2}} \int\limits_{0}^{\varphi (\theta)} \frac{|F|}{|\widetilde{F}|} q(t) dt d\theta  = \frac{|F|}{|\widetilde{F}|} |\widetilde{F}_i|.
\end{align*}
If equality holds in the inequality above, then $\varphi(\theta) = {t_0}$ except possibly on a subset of spherical measure zero in $\S^{n-2}$. Since $F_i \subseteq F$ is compact and convex, the equality is only possible if $F_i = F$.

\bigskip

\appendix

\section{Reverse isoperimetric inequality for $\lambda$-convex curves in $\mathbb R^2$, revisited}
\label{App:RIPDim2}

In this appendix, we give an alternative proof of the result in \cite{BorDr14} and \cite{FKV} on the reverse isoperimetric problem in $\R^2$ by using the method exploited in this paper (see Section~\ref{Sec:ProofA}).

\begin{theorem}[$2$-dimensional reverse isoperimetric inequality]
\label{AppThm:A}
Let $K \subset \R^2$ be a $\lambda$-convex body and $L \subset \R^2$ be a $\lambda$-convex lens. If $|\partial K| = |\partial L|$, then $|K| \ge |L|$, and the equality is only possible if $K$ is a lens.
\end{theorem}

\begin{proof}[Proof of Theorem~\ref{AppThm:A}]
As in the proof of Theorem~\ref{mainthm}, using Blaschke's rolling theorem, we can approximate the $\lambda$-convex body $K$ by $\lambda$-convex polygons. We will establish the required inequality for polygons, and the inequality in the general case follows by the limiting argument as in Section~\ref{SSec:StepII}. The equality case follows exactly as in the proof of Theorem~\ref{mainthm} (see also the remark right after the current proof for an outline of an alternative approach to the equality case). 

The result is immediate when $K$ is a $\lambda$-convex lens, including the ball. Hence, below we assume that $K$ is not a $\lambda$-convex lens.

Assume now that $K$ is a $\lambda$-convex polygon. By Theorem~\ref{Thm:Ma},
\[
|K| = \int_{0}^{r(K)} |\partial K_t| dt.
\]

As in Section~\ref{SSec:StepI}, to establish the result for $\lambda$-convex polygons, it is enough to show that one-sided derivatives at $0$ satisfy
\begin{equation}
\label{Eq:Derivative}
\left.\frac{d}{dt}\right|_{t=0^+} |\partial K_t| \ge \left.\frac{d}{dt}\right|_{t=0^+} |\partial L_t|, \text{ with equality if and only if } K \text{ is a $\lambda$-convex lens}.
\end{equation}
Here $L_t$ is the corresponding inner parallel body at distance $t$ for the lens $L$ isoperimetric to $K$:
\begin{equation}
\label{Eq:Isoperimetric}
|\partial K| = |\partial K_0| = |\partial L_0| = |\partial L|.
\end{equation}
Note that $L_t$ is a $(1/\lambda - t)^{-1}$-convex lens.

Without loss of generality, assume $\lambda=1$. Denote by $\beta_i$, $i \in \{1, \ldots, m\}$, the angles that support the arcs of unit circles forming $\partial K$, and denote by $\gamma_i$ the exterior angles between the adjacent arcs at the vertices of $K$. Then 
\[
|\partial K| =  \sum\limits_{i=1}^m \beta_i
\]
and by the remark after Lemma~\ref{surf_area_derivative},
\begin{equation}
\label{Eq:2D}
\left.\frac{d}{dt}\right|_{t=0^+} |\partial K_t| = - |\partial K| - 2 \cdot \sum\limits_{i=1}^m \tan \frac{\gamma_i}{2}.
\end{equation}
In particular, in the case of the lens, this expression reads as
\[
\left.\frac{d}{dt}\right|_{t=0^+} |\partial L_t| = - |\partial L| - 4 \tan \frac{\gamma_*}{2},
\]
where $\gamma_*$ is the angle at the vertex of the lens.

Under assumption~\eqref{Eq:Isoperimetric}, in order to establish \eqref{Eq:Derivative}, we need to show that
\begin{equation}
\label{Eq:Goal}
\sum\limits_{i =1}^m \tan \frac{\gamma_i}{2} \le 2 \tan \frac{\gamma_*}{2}, \quad \text{ with equality only if } m = 2 \text{ and } \gamma_1 = \gamma_2 = \gamma_*.
\end{equation}

Observe that we have the following constraints:
\begin{equation}
\label{Eq:Constraints}
\max \gamma_i \le \gamma_* \qquad \textup{and} \qquad  \sum\limits_{i = 1}^m \gamma_i = 2 \gamma_*.
\end{equation}
The first constraint follows from Blaschke's rolling theorem since the two supporting circles at a vertex with angle $\gamma_i$ form a lens of length $2\pi - 2\gamma_i$ that contains $K$, and the inclusion implies that $2\pi - 2\gamma_* = |\partial L| = |\partial K| \le 2\pi - 2\gamma_i$. The second constraint follows from the fact that the spherical image of a convex curve has length $2\pi$ and is a sum of the spherical images of the smooth boundary components (summing up to the length of the boundary in our case) plus the angles at the vertices.

Now, using \eqref{Eq:Constraints} and the fact that $(\tan x)/x$ is an increasing function on $[0, \pi/2)$, we estimate:
\begin{equation}
\label{Eq:Comp}
\sum\limits_{i=1}^m \tan{\frac{\gamma_i}2} = \sum\limits_{i=1}^m \frac{\gamma_i}{2} \cdot \frac{\tan{\frac{\gamma_i}{2}}}{\frac{\gamma_i}{2}} \le \frac{\tan{\frac{\gamma_*}{2}}}{\frac{\gamma_*}{2}} \sum\limits_{i=1}^m \frac{\gamma_i}{2} = 2\tan{\frac{\gamma_*}{2}}.
\end{equation}
Thus, \eqref{Eq:Goal} follows. 
\end{proof}

\begin{remark}
In this remark, we outline an alternative approach to control the equality case in the limit of approximating $1$-convex polygons in $\R^2$. Some details of the proofs are left to the reader.

Let $K \subset \R^2$ be a $1$-convex body. We assume that $K$ is not a $1$-convex polygon. Our goal is to show that if $L$ is a $1$-convex lens with $|\partial L| = |\partial K|$, then $|K| > |L|$. From the approximation by $1$-convex polygons and the proof above, we know that $|K| \ge |L|$.

Let $\mathcal V(K)$ be the set of vertices of $\partial K$. This is a countable set, and let us denote by $\gamma_v$ the outer angle of $\partial K$ at the vertex $v \in \mathcal V(K)$. 

One can show that 
\begin{equation}
\label{App:Eq:2D2}
\left.\frac{d}{dt}\right|_{t=0^+} |\partial K_t| = -2\pi - \sum_{v \in \mathcal V(K)}\left(2 \tan \frac{\gamma_v}{2} - \gamma_v\right).
\end{equation}
When $K$ is a $1$-convex polygon, it follows from \eqref{Eq:2D} by noting that in $\mathbb R^2$ one has $|\partial K| = 2\pi - \sum_{i=1}^m \gamma_i$. For general convex bodies in $\mathbb R^2$ (not necessarily $1$-convex), this formula can be justified using support functions; in particular, the sum of the positive terms $2 \tan(\gamma_v/2) - \gamma_v$ is finite.

In particular, if $K$ has no vertices (and hence, $\partial K$ is $C^1$-smooth), then the derivative in \eqref{App:Eq:2D2} is equal to $-2\pi$.

Let $(K^j)$ be a monotone-decreasing sequence of $1$-convex polygons approximating $K$ (Proposition~\ref{Prop:Approx}). Let $L^j$ be the corresponding sequence of $1$-convex lenses so that
\[
|\partial K^j| = |\partial L^j| \quad \text{for every }j.
\]
By continuity and by normalizing the sequence $(L^j)$ appropriately, it follows that $L^j \to L$, where $L$ is the lens such that $|\partial K| = |\partial L|$. 

Furthermore, we can choose the approximating sequence $(K^j)$ so that:
\begin{enumerate}
    \item the sets $\mathcal V(K^j) \cap \mathcal V(K)$ are non-decreasing and exhaustive, i.e.,
    \[
    \bigcup_{j=1}^\infty\mathcal V(K^j) \cap \mathcal V(K) = \mathcal V(K)
    \]
    (i.e., the approximating polygons share more and more vertices with $K$);
    \item for every $v \in \mathcal V(K^j) \cap \mathcal V(K)$, the sides of $K^j$ making the vertex $v$ are arcs of the extremal supporting unit circles of $K$ at $v$;
    \item the angles at the remaining vertices $\mathcal V(K^j) \setminus \mathcal V(K)$ tend to zero:
    \[
    \max\limits_{v \in \mathcal V(K^j) \setminus \mathcal V(K)} \gamma_v =:\alpha^j \searrow 0.
    \]
\end{enumerate}

For the approximating polygon $K^j$, using \eqref{App:Eq:2D2}, we can write
\[
\left.\frac{d}{dt}\right|_{t=0^+} |\partial K^j_t| = -2\pi - \sum_{v \in \mathcal V(K^j) \cap \mathcal V(K)}\left(2 \tan \frac{\gamma_v}{2} - \gamma_v\right) - \sum_{v \in \mathcal V(K^j) \setminus \mathcal V(K)}\left(2 \tan \frac{\gamma_v}{2} - \gamma_v\right).
\]
As in the proof above, the second sum can be estimated as
\[
0 \le \sum_{v \in \mathcal V(K^j) \setminus \mathcal V(K)}\left(2 \tan \frac{\gamma_v}{2} - \gamma_v\right) \le \left(\frac{\tan \frac{\alpha^j}{2}}{\frac{\alpha^j}{2}} - 1\right) \sum_{v \in \mathcal V(K^j) \setminus \mathcal V(K)}\gamma_v \le 2 \pi\left(\frac{\tan \frac{\alpha^j}{2}}{\frac{\alpha^j}{2}} - 1\right).
\]
Since $\alpha^j \to 0$ as $j \to \infty$ (by property (3) of the approximating sequence), the second sum converges to $0$. Using a similar estimate and properties (1)--(2) of the approximating sequence, one can show that the first sum converges to $\sum_{v \in \mathcal V(K) }\left(2 \tan (\gamma_v/2) - \gamma_v\right)$. Together, we conclude that
\begin{equation}
    \label{App:Eq:Conv}
    \lim_{j \to \infty} \left(\left.\frac{d}{dt}\right|_{t=0^+} |\partial K^j_t| \right)= \left.\frac{d}{dt}\right|_{t=0^+} |\partial K_t|.
\end{equation}
Likewise,
\begin{equation}
    \label{App:Eq:Conv2}
    \lim_{j \to \infty} \left(\left.\frac{d}{dt}\right|_{t=0^+} |\partial L^j_t| \right)= \left.\frac{d}{dt}\right|_{t=0^+} |\partial L_t|.
\end{equation}

Let $\gamma_*^j$ be the angle of $L^j$. By construction, 
\[
\gamma_*^j \nearrow \gamma_*.
\]

Now, assume $\mathcal V(K) \neq \emptyset$ and let $\alpha_* := \max \limits_{v \in \mathcal V(K)} \gamma_v$ be the maximal angle of $\partial K$. Then there exists $N>0$ so that for $j \ge N$,
\[
\max_{v \in \mathcal V(K^j)} \gamma_v=\alpha_*.
\]
Since $K$ and none of $K^j$'s (for sufficiently large $j$) is a lens, we also have  
\begin{equation}
    \label{App:Eq:B}
0<\alpha_* < \gamma^j_* < \gamma_*    
\end{equation}
(compare \eqref{Eq:Constraints}). We claim that for all sufficiently large $j$ (larger than $N$), 
\begin{equation}\label{Eq:ClaimApp}
	\left.\frac{d}{dt}\right|_{t=0^+} |\partial K^j_t| \ge \left.\frac{d}{dt}\right|_{t=0^+} |\partial L^j_t| + C,
\end{equation}
where $C=C(\gamma_*, \alpha_*) > 0$. First, observe that a more detailed analysis of \eqref{Eq:Comp} gives the following bound:
\[
\left.\frac{d}{dt}\right|_{t=0^+} |\partial K^j_t| \ge \left.\frac{d}{dt}\right|_{t=0^+} |\partial L^j_t| + \mathcal E_j,
\]
where 
\[
\mathcal E_j := \gamma_*^j \cdot \left(\frac{\tan \frac{\gamma_*^j}{2}}{\frac{\gamma_*^j}{2}} - \frac{\tan \frac{\alpha^j_*}{2}}{\frac{\alpha^j_*}{2}}\right) > 0, \qquad \alpha^j_* :=\max_{v \in \mathcal V(K^j)} \gamma_v.
\]
By construction, 
\[
\lim_{j \to \infty} \mathcal E_j = \gamma_* \cdot \left(\frac{\tan \frac{\gamma_*}{2}}{\frac{\gamma_*}{2}} - \frac{\tan \frac{\alpha_*}{2}}{\frac{\alpha_*}{2}}\right) >0,
\]
and the inequality is strict because of \eqref{App:Eq:B}. The claim follows.

Passing to the limit in \eqref{Eq:ClaimApp} using~\eqref{App:Eq:Conv} and \eqref{App:Eq:Conv2}, we obtain in the case $\mathcal V(K) \neq \emptyset$ that
\[
\left.\frac{d}{dt}\right|_{t=0^+} |\partial K_t| \ge \left.\frac{d}{dt}\right|_{t=0^+} |\partial L_t| + C > \left.\frac{d}{dt}\right|_{t=0^+} |\partial L_t|.
\]
If $\mathcal V(K) = \emptyset$, then the same conclusion follows immediately without the approximation because
\[
\left.\frac{d}{dt}\right|_{t=0^+} |\partial K_t| = -2\pi > -2\pi - 2\left(2\tan \frac{\gamma_*}{2} - \gamma_*\right) = \left.\frac{d}{dt}\right|_{t=0^+} |\partial L_t|.
\]
Hence, in both cases, we obtain $|\partial K_t| > |\partial L_t|$ in some small right neighborhood of $t = 0$. Integrating as earlier, we conclude that $|K| > |L|$. This finishes the remark.
\end{remark}

\begin{remark}
In $\R^2$ and $\R^3$, it is possible to explicitly compute the volume of a $\lambda$-convex lens as a function of its surface area. 
Using these explicit formulas, we obtain the following reverse isoperimetric inequalities (RIP) as corollaries to Theorem~\ref{mainthm} and Theorem~\ref{AppThm:A}:

\begin{corollary}[Explicit RIP in dimension 2]
If $K \subset \R^2$ is a $\lambda$-convex body, then
\[
\quad |K| \ge \frac{|\partial K|}{2\lambda} - \frac{1}{\lambda^2} \sin \left(\frac{|\partial K| \lambda}{2}\right),
\]
and equality holds if and only if $K$ is a $\lambda$-convex lens.
\end{corollary}

\begin{corollary}[Explicit RIP in dimension 3]
If $K \subset \R^3$ is a $\lambda$-convex body, then
\[
\quad |K| \ge \frac{|\partial K|^2 \lambda}{96\pi^2} \cdot  \left(12\pi - |\partial K|\lambda^2 \right),
\]
and equality holds if and only if $K$ is a $\lambda$-convex lens.
\end{corollary}
\end{remark}


\section{Lenses vs.\ spindles, an asymptotic competition}
\label{App:Ass}

Here we discuss extra evidence in support of the Borisenko conjecture. By Theorem~\ref{AppThm:A}, a $\lambda$-convex lens is the solution for the $2$-dimensional reverse isoperimetric problem in $\R^2$. If one tried to conjecture a potential solution for dimension $n \ge 3$, a natural candidate would be a body obtained by rotating this lens. However, there are two ways one can rotate it in order to preserve convexity. The result of rotating the $2$-dimensional lens with respect to the axis of symmetry which realizes the smallest width is the $n$-dimensional lens. By Theorem~\ref{mainthm}, this body is indeed the solution in $\R^3$. 
On the other hand, one can rotate the lens with respect to the axis of symmetry which realizes its diameter. The resulting body, which we call a \emph{spindle} and define more formally below, is a natural `competitor' to the lens for the optimal shape. As we will see below, in $\mathbb R^n$, the lens `wins' this competition asymptotically as $n \to \infty$. Spindles also appear as solutions to other optimization problems with curvature constraints (see, e.g., \cite{BezdekConj, BorDr13}). 

We say that a convex body $\textup{Sp} \subset \R^n$ is a \emph{$\lambda$-convex spindle} if $\textup{Sp}$ is the intersection of all balls of radius $1/\lambda$ containing a pair of given points. The boundary of $\textup{Sp}$ is obtained by taking a smaller arc of a circle of radius $1\slash \lambda$ in the $x_1 x_n$-plane, with both endpoints on the $x_1$-axis, and rotating it in $\R^n$ about the $x_1$-axis (here and below, $x_1, \ldots, x_n$ are the Cartesian coordinates in $\R^n$).

The next proposition shows that asymptotically a lens has smaller volume than a spindle, assuming that their surface areas are the same. 

Without loss of generality, everywhere below we assume $\lambda = 1$.

\begin{proposition}[Lens wins at infinity]
    Let $h_2 \in (0,1)$. For each $n\ge 2,$ let  $L_n\subset \mathbb{R}^n$ be the 1-convex lens with the circumradius $h_2$. Then there exists $n_0=n_0(h_2)$ such that for every $n\ge n_0$, there exists a $1$-convex spindle $\textup{Sp}_n \subset \mathbb{R}^n$  with  $|\partial L_n| = |\partial\textup{Sp}_n|$ and   $|L_n| < |\textup{Sp}_n|$.
\end{proposition}

\begin{proof}
We start by computing the surface areas of a lens and a spindle. Consider the circle of radius $1$ in the $x_1x_n$-plane centered at the point $(0, \dots, 0, -h_1)$ for $0 < h_1 < 1$. Let $C_1$ be the arc of this circle above the $x_1$-axis given by the following equation
	$$
	x_n = \sqrt{1 - x_1^2} - h_1,
	$$
	where $x_1$ changes between $0$ and $\sqrt{1 -h_1^2}$.
 By rotating $C_1$ about the $x_1$-axis, we obtain the part of $\partial \textup{Sp}$ in the half-space $x_1 \ge 0$. Now we compute the surface area of $\textup{Sp}$ using the corresponding formula for the surface area of the body of revolution:
\begin{align*}
	|\partial \textup{Sp}| &= 2 (n-1) |\B^{n-1}| \int\limits_{0}^{\sqrt{1-h_1^2}} \left(\sqrt{1 - x_1^2} - h_1\right)^{n-2} \frac{1}{\sqrt{1-x_1^2}}\, dx_1 \\
	&= 2 (n-1) |\B^{n-1}| \int\limits_{0}^{\sqrt{1-h_1^2}} e^{ - n \varphi_1(x_1)} f_1(x_1)\, dx_1 
\end{align*}
where $\varphi_1(x_1) = - \ln(\sqrt{1-x_1^2} - h_1)$ and $f_1(x_1) = \frac{1}{\sqrt{1-x_1^2} \left(\sqrt{1 - x_1^2} - h_1\right)^2}$.
We will  use Laplace's method to obtain the asymptotic behavior of $|\partial \textup{Sp}|$ as $n \rightarrow \infty$. Since $\varphi_1(x_1)$ is a monotonically increasing function on $[0, \sqrt{1-h_1^2}]$ that attains its absolute minimum at the endpoint $x_1 = 0$, we apply \cite[Eq. (5.1.17)]{BH} to get  
\begin{align*}
	\int\limits_{0}^{\sqrt{1-h_1^2}} e^{ - n \varphi_1(x_1)} f_1(x_1)\, dx_1 &=  \sqrt{\frac{\pi}{2n \varphi_1''(0)}} f_1(0) e^{-n \varphi_1(0)} + \left(\frac{f_1'(0)}{\varphi''_1(0)} -\frac{f_1(0)\varphi'''_1(0)}{3(\varphi''_1(0))^2}  \right)\frac{e^{-n\varphi_1(0)}}{n} + \\
	&+ O\left(\frac{e^{-n \varphi_1(0)}}{n^{\frac32}}\right) = \sqrt{\frac{\pi}2} (1-h_1)^{n - \frac32} \frac1{\sqrt{n}} + O\left(\frac{(1-h_1)^n}{n^\frac32}\right).
\end{align*}
Here we used that both $\phi_1$ and $f_1$ are even, $f_1'(0)=\phi_1'''(0)=0$, so that the term with $\frac1n$ vanishes.

Therefore,
\begin{align*}
	|\partial \textup{Sp}| &=  \sqrt{2\pi} (n-1) |\B^{n-1}|  (1-h_1)^{n-\frac32} \frac{1}{\sqrt{n}} \left(1 + O\left(\frac{(1-h_1)^\frac32}n\right)\right).
\end{align*}

Now we can find the surface area of a lens $L$. For this, consider the circle of radius $1$ in the $x_1x_n$-plane centered at the point $(-\sqrt{1 - h_2^2}, \dots, 0, 0)$ for $0 < h_2 < 1$. Let $C_2$ be the arc of this circle  given by the following equation  
$$
x_n = \sqrt{1 - \left(x_1+\sqrt{1 - h_2^2}\right)^2}
$$
where $x_1$ changes between $0$ and $1-\sqrt{1 - h_2^2}$. The surface area of $L$ is 
\begin{align*}
	|\partial L| &= 2 (n-1) |\B^{n-1}| \int\limits_{0}^{1-\sqrt{1 - h_2^2}} \left(1 - \left(x_1 + \sqrt{1 - h_2^2}\right)^2 \right)^{\frac{n-3}2} dx_1  \\
	&= 2 (n-1) |\B^{n-1}| \int\limits_{0}^{1-\sqrt{1 - h_2^2}} \varphi_2(x_1)^n f_2(x_1)\, dx_1
\end{align*}
where $\varphi_2(x_1) = \sqrt{1 - \left(x_1 + \sqrt{1 - h_2^2}\right)^2}$ and $f_2(x_1) = \frac{1}{\left(1 - \left(x_1 + \sqrt{1 - h_2^2}\right)^2\right)^{\frac32}}$. Since $\varphi_2(x_1)$ is a monotonically decreasing function on $[0, 1-\sqrt{1 - h_2^2}]$ that attains its absolute maximum at the endpoint $x_1 = 0$, we can apply \cite[Eq. (5.6.23)]{BH} to get
\begin{align*}
	\int\limits_{0}^{1-\sqrt{1 - h_2^2}} \varphi_2(x_1)^n f_2(x_1)\, dx_1 &=  \frac{h_2^{n-1}}{\sqrt{1-h_2^2}} \frac1n \left(1 + O\left(\frac1n\right)\right).
\end{align*}  
Hence,
\begin{align*}
	|\partial L| &= 2 (n-1) |\B^{n-1}| \frac{h_2^{n-1}}{\sqrt{1-h_2^2}}  \frac1n \left(1 + O\left(\frac1n\right)\right)
\end{align*}
as $n \to \infty$.

In the same manner, we obtain the expressions for volumes:
\begin{align*}
	|\textup{Sp}| &=  \sqrt{2\pi} |\B^{n-1}| (1-h_1)^{n-\frac12} \frac1{\sqrt{n}} \left(1 + O\left(\frac1n\right)\right); \\
	|L| &=   2 |\B^{n-1}| \frac{h_2^{n+1}}{\sqrt{1-h_2^2}}  \frac1n \left(1 + O\left(\frac1n\right)\right),
\end{align*}
where $h_1, h_2 \in (0,1)$.

Now we fix $h_2 \in (0,1)$. Denote by $|\partial\textup{Sp}_n(h_1)|$ the surface area of the spindle with parameter $h_1\in(0,1)$. We note that $|\partial\textup{Sp}_n(h_1)|$ is strictly decreasing in $h_1$, $|\partial\textup{Sp}(0)| = |\partial \B^{n}|$ and $|\partial\textup{Sp}(1)| = 0.$ Since $0 < |\partial L_n| <  |\partial \B^{n}|$, then there exists a unique $h_1(n) \in (0,1)$ such that $|\partial\textup{Sp}_n(h_1(n))| = |\partial L_n|.$

Now we show that $\lim\limits_{n \to \infty} h_1(n) = 1- h_2$. Choose fixed numbers $h_-$ and $h_+$ such that $0<h_-<1-h_2<h_+<1.$ Then 
$$
\frac{|\partial\textup{Sp}_n(h_1)|}{|\partial L_n|} = 
\sqrt{\frac{\pi}{2}} \sqrt{n} \sqrt{1-h_2^2} \,
\frac{(1-h_1)^{n-\frac32}}{h_2^{n-1}}
\left(1+O\left(\frac1n\right)\right)
$$
for every fixed $h_1 \in(0,1).$
Since $1-h_{-} > h_2$, it follows that
$$
\frac{|\partial\textup{Sp}_n(h_-)|}{|\partial L_n|}\longrightarrow\infty.
$$
Similarly, since $1-h_+<h_2$, then
$$
\frac{|\partial\textup{Sp}_n(h_+)|}{|\partial L_n|}\longrightarrow 0.
$$
Therefore, 
$$
|\partial\textup{Sp}_n(h_-)|>|\partial L_n| = |\partial\textup{Sp}_n(h_1(n))| >|\partial\textup{Sp}_n(h_+)|
$$
for all sufficiently large $n$. Since $|\partial\textup{Sp}_n(h_1(n))|$ is strictly decreasing in $h_1$, we get $h_-<h_1(n)<h_+.$
Also, the numbers $h_-$ and $h_+$  were arbitrary, so they can be chosen arbitrarily close to $1-h_2$, which implies that  $\lim\limits_{n \to \infty} h_1(n) = 1- h_2$.

Since $\lim\limits_{n \to \infty} h_1(n) = 1- h_2$, the sequence $h_1(n)$ stays away from $0$ and $1$ for all sufficiently large $n.$ On such a fixed range of parameters, the quantities appearing in the application of Laplace's method remain bounded and nondegenerate, and hence the constants in the error estimates can be chosen uniformly over this range. Thus, the above asymptotics are valid when $h_1=h_1(n).$

Using $|\partial\textup{Sp}_n|=|\partial L_n|,$
we obtain
$$
\frac{|\textup{Sp}_n|}{|L_n|} = \frac{{|\textup{Sp}_n|}/{|\partial\textup{Sp}_n|}}{{|L_n|}/{|\partial L_n|}} = \frac{1 - h_1(n)}{h_2^2} \left(1 + O\left(\frac1n\right)\right) \longrightarrow
\frac1{h_2}>1,
$$
where we used that $\lim\limits_{n \to \infty} h_1(n) = 1- h_2$.
Thus, $|L_n| < |\textup{Sp}_n|$ for all sufficiently large $n$.
\end{proof}

\end{document}